\documentclass[11pt,A4paper]{article}

\usepackage[left=30mm,right=30mm,top=25mm,bottom=30mm]{geometry}

\usepackage[utf8]{inputenc}
\usepackage[T1]{fontenc}
\usepackage[main=english,french]{babel}

\usepackage{mathpazo}
\usepackage{microtype}
\usepackage[dvipsnames]{xcolor}

\usepackage{graphicx}
\graphicspath{{images/}}
\usepackage{caption}
\usepackage{subcaption}
\usepackage{float}

\usepackage{amsmath,amssymb,amsthm}
\usepackage{amsbsy}
\usepackage{bm}
\usepackage{mathtools}
\usepackage{mathrsfs}

\usepackage{tikz}
\usepackage{tikz-cd}

\usepackage{enumitem}
\usepackage{titlesec}
\usepackage{sectsty}

\renewcommand{\thesection}{\arabic{section}}
\renewcommand{\thesubsection}{\thesection.\arabic{subsection}}
\renewcommand{\thesubsubsection}{\thesubsection.\arabic{subsubsection}}
\renewcommand{\theparagraph}{\thesubsubsection.\arabic{paragraph}}

\titleformat{\section}
  {\large\bfseries}
  {\thesection.\quad}
  {0pt}
  {}
\titleformat{\subsection}
  {\normalsize\bfseries}
  {\thesubsection.\quad}
  {0pt}
  {}
\titleformat{\subsubsection}
  {\normalsize\bfseries}
  {\thesubsubsection.\quad}
  {0pt}
  {}
\titleformat{\paragraph}[runin]
  {\normalsize\bfseries}
  {\theparagraph.\quad}
  {0pt}
  {}
  [.]

\titlespacing{\section}{0pt}{12pt}{0pt}
\titlespacing{\subsection}{0pt}{6pt}{0pt}
\titlespacing{\subsubsection}{0pt}{6pt}{0pt}
\titlespacing*{\paragraph}{0pt}{6pt}{6pt}

\sectionfont{\large \bfseries}
\subsectionfont{\normalsize}

\usepackage[hang,flushmargin]{footmisc}

\makeatletter
\long\def\symbolfootnote[#1]#2{\begingroup%
\def\thefootnote{\fnsymbol{footnote}}\footnote[#1]{#2}\endgroup}
\def\blfootnote{\xdef\@thefnmark{}\@footnotetext}
\makeatother

\usepackage[sort,nocompress]{cite}

\usepackage{color}
\definecolor{linkred}{RGB}{148,33,147}
\definecolor{linkblue}{RGB}{16,78,139}

\usepackage[
  pdftitle={Counting and entropy for hyperbolic surface amalgams},
  pdfauthor={Hugo Parlier and Yandi Wu},
  ocgcolorlinks,
  linkcolor=linkred,
  citecolor=linkred,
  urlcolor=linkblue
]{hyperref}

\usepackage{cleveref}

\makeatletter
\long\def\@footnotetext#1{%
\H@@footnotetext{%
\ifHy@nesting
\hyper@@anchor{\@currentHref}{#1}%
\else
\Hy@raisedlink{\hyper@@anchor{\@currentHref}{\relax}}#1%
\fi
}}
\def\@footnotemark{%
\leavevmode
\ifhmode\edef\@x@sf{\the\spacefactor}\nobreak\fi
\H@refstepcounter{Hfootnote}%
\hyper@makecurrent{Hfootnote}%
\hyper@linkstart{link}{\@currentHref}%
\@makefnmark
\hyper@linkend
\ifhmode\spacefactor\@x@sf\fi
\relax
}
\makeatother

\usepackage{yfonts}
\usepackage{listings}
\usepackage{comment}

\theoremstyle{plain}
\newtheorem{theorem}{Theorem}[section]
\newtheorem{corollary}[theorem]{Corollary}
\newtheorem{lemma}[theorem]{Lemma}

\newtheorem{proposition}[theorem]{Proposition}

\theoremstyle{definition}
\newtheorem{assumption}[theorem]{Assumption}
\newtheorem{definition}[theorem]{Definition}
\newtheorem{remark}[theorem]{Remark}

\newtheorem{observation}[theorem]{Observation}

\DeclareMathOperator{\arcsinh}{arcsinh}
\DeclareMathOperator{\arccosh}{arccosh}

\DeclareMathOperator{\Area}{Area}
\DeclareMathOperator{\sys}{sys}
\DeclareMathOperator{\vol}{vol}
\DeclareMathOperator{\Bad}{Bad}
\DeclareMathOperator{\Good}{Good}
\newcommand{\G}{{\mathcal G}}

\newcommand{\N}{{\mathbb N}}

\title{Counting and entropy for hyperbolic surface amalgams}
\author{Hugo Parlier and Yandi Wu}
\date{}

\begin{document}

\maketitle

\begin{abstract}
    This paper is about closed hyperbolic surface amalgams with a focus on the growth of the number of closed geodesics. As in the case of surfaces, we show that topological and volume entropies coincide, but we show stark differences in how they behave according to geometric data with upper and lower bounds on the number of closed geodesics which depend on the length of the systole and the length of the pasting curves. In particular, we show that the entropy can increase exponentially in terms of the pasting length in the absence of a lower bound on the systole.
\end{abstract}

\section{Introduction} 
\label{sec:intro}

For many reasons, the study of the growth of the number of closed geodesics is central to the study of hyperbolic surfaces and their moduli spaces. This topic and related questions exemplify the relationship between geometry and dynamics. In this paper, we will study a class of objects that are natural generalizations of surfaces, defined as "2-dimensional P-manifolds" in \cite{lafont} and as "hyperbolic graph surfaces" in \cite{buyalo}. 

Closed (orientable) hyperbolic surfaces can be constructed by pasting pairs of pants along their cuffs, under the restriction that the boundary curves are pasted in pairs of equal length. In our case, we lift the restriction on the number of curves pasted together, and so while we no longer have surfaces, we obtain closed geometric objects which we call hyperbolic surface amalgams. We consider a particular class of them, which we call proper surface amalgams, where we've ruled out surface amalgams with boundary, as well as surface amalgams which are "just" closed hyperbolic surfaces.

Huber's prime geodesic theorem for closed hyperbolic surfaces gives a precise asymptotic growth function for the number of primitive closed geodesics of length less than $L$. A similar result holds for surface amalgams (see Theorem \ref{thm:entropy}); the number of curves of a proper surface amalgam $X$ grows asymptotically like $\frac{e^{h L}}{h L}$, where $h$ is either the topological entropy of the geodesic flow or the volume entropy of $X$. The fact that the two types of entropies are equal is an application of work of Leuzinger \cite{leuzinger}, and then the growth follows from a more general theorem of Ricks \cite{rickscounting}; see \Cref{subsec:entropy}. Buyalo \cite{buyalo} had previously studied the volume entropy for proper surface amalgams, showing that $h_X>1$, unlike for closed hyperbolic surfaces where the entropy is always $1$. In addition, Buyalo shows that $h_X$ tends to infinity as a function of the length of the singular curves; that is, the pasting curves along which we have glued more than two copies of a curve. 

Our main contributions are quantified upper and lower bounds on the size of $\G_X(L)$, the set of primitive closed geodesics of length at most $L$ on $X$. In addition to $L$ and the area of the surface amalgam, our upper bound also involves the total length of the gluing curves $B$ and $\sys(X)$, the length of the systole of $X$, that is, the length of the shortest non-trivial closed geodesic. The systole is taken into account in following quantity
$$
r_0:= r_0(X) = 
\min \bigg\{\frac{\log(3)}{4}, \frac{\sys(X)}{4}\bigg\}.
$$
which we use throughout the paper. Our upper bound can be stated as:
\begin{theorem}\label{thm:mainupper}
Let $X$ be proper surface amalgam of area $A$, of total gluing length $B$ and let $r_0$ be as above. Then
\[
\#\G_X(L)\ \le\ \frac{4A}{\pi r_0^{2}}\left(15+\frac{3AB}{4\pi r_0}\right)^{\frac{25\,AB}{\pi r_0^{2}}(L+r_0)}.
\]
As a consequence, $h(X)$, the entropy of $X$, satisfies
$$
h(X) \leq \frac{25\,AB}{\pi r_0^{2}} \log\left( 15+\frac{3AB}{4\pi r_0} \right).
$$
\end{theorem}

We think of this result as the corresponding result to a widely used result of Buser's \cite[Lemma 6.6.4]{buser} for closed surfaces. Buser's result provides a bound which only depends on $L$ and on the area of the surface, whereas we require more geometric data. Note that Buyalo's results show that any upper bound {\it does} require a condition on the gluing length $B$ \cite{buyalo}, but here we also use a bound on the systole. Note however, that for fixed area, and lower bound on systole, the upper bound behaves like $e^{C B \log(B)L}$ for some constant $C>0$. 

We then provide a series of constructions of surface amalgams which show that there is, at least some, dependency in $B$ and $\sys(X)$. 

Our first example is a very natural example of surface amalgam which comes from closed surfaces. Take a closed hyperbolic surface $S$ and choose a simple closed geodesic $\beta$ on it. By taking multiple copies of $S$ (at least $2$) and by pasting them along $\beta$ by the identity, we obtain a surface amalgam $X=X_{S,\beta,m}$, parametrized by our choice of $S$, $\beta$ and $m\geq 2$, the number of copies of $S$. We show the following:

\begin{theorem}\label{thm:C}
There exists a constant $c$ that only depends on the topology of $S$ such that for sufficiently large $L$,
\[
\#\G_X(L) > e^{c(S) \ell(\beta) L}.
\] 
\end{theorem}
The proof provides explicit constants, and the technical statement is Proposition \ref{prop:longbeta}. While our proof uses geodesic currents and equidistribution results for (long) closed geodesics on $S$, the idea can be explained rather simply. First note that the result is only meaningful if $\beta$ is sufficiently long. Now if closed geodesics on $S$ are sufficiently long, they start to equidistribute in the unit tangent bundle, meaning that they begin to regularly intersect $\beta$, and their intersection with $\beta$ is (directly) proportional to their length in a proportion that only depends on the length of $\beta$. The key is that these geodesics on $S$ lift to multiple geodesics on $X$. The result comes from quantifying these different phenomena. 

We note that a result of Ledrappier and Lim \cite{LedrappierLim2010} provides a related result which applies to some very special surface amalgams constructed for \Cref{thm:B}. For example, consider a genus $2$ surface obtained by reflecting a one-holed torus along its boundary curve which becomes a separating curve of the genus $2$ surface. Now consider the surface amalgam obtained by pasting two copies of the surface along the separating curve, say $\beta$. The universal cover of this surface amalgam is a regular building, and Ledrappier and Lim study its resulting volume entropy and show that it has a (strict) lower bound \cite[Theorem 1.4]{LedrappierLim2010} which is also linear in $\ell(\beta)$. By the previous results, this volume entropy also gives a lower bound on the number of geodesics which is slightly stronger for this particular surface amalgam. 

In any event, these theorems, as well as the results Buyalo \cite{buyalo}, all tell us that we can make the entropy arbitrarily large, but without providing explicit lower bounds on the number of closed geodesics for any $L$ or any insight into the dependency on the length of the systole. And put together with Theorem \ref{thm:mainupper}, Theorem \ref{thm:C} implies that, for fixed lower bound on the systole, the entropy is lower bounded by a linear function of $B$ and upper bounded by a linear function of $B\log(B)$. Our next construction shows that this is far from being the case in general.

We then provide an explicit construction of surface amalgams $X_B$, of gluing length $B$ which, for sufficiently large $B$, have many closed geodesics. They are all homeomorphic to a genus four-holed sphere with all four boundary curves pasted together.

\begin{theorem}\label{thm:B}
For any $B\geq 2 \log(24)$, there exist surface amalgams $X_B$ of area $4\pi$ and of pasting length $B$ that satisfy
\[
\#\G_{X_B}(L) \;\geq \;2^{\,1+\frac{1}{96}\frac{e^{B/2}}{B}\,L}
\]
for all $L\geq 2B + 5$. In particular, the entropies of $X_B$ satisfy
\[
h(X_B) \geq c\,\frac{e^{B/2}}{B}
\]
for some universal constant $c>0$.
\end{theorem}
The constant $c$ is explicit (it can be taken to be $\frac{\log 2}{96}$) but no attempt to optimize has been made. In particular, In light of Theorem \ref{thm:mainupper}, their systole necessarily goes to $0$, and in fact this a feature of the construction. The exponential growth of the entropy in terms of $B$ shows that making the systole small can have a much greater effect on entropy than "just" by increasing the pasting length. Again, to show the existence of a large family of closed geodesics, as in Theorem \ref{thm:C}, the proof uses multiple lifts of the same geodesic.

\noindent \textbf{Outline of the paper.} We now outline the paper. \Cref{sec:prelims} establishes some useful facts about hyperbolic surfaces and proper hyperbolic surface amalgams which will be used in the remainder of the paper. Additionally, \Cref{subsec:entropy} uses results from the literature to establish the relationship between entropy and counting for surface amalgams. Finally, in \Cref{sec:counting}, we prove our main results, \Cref{thm:mainupper}, \ref{thm:C}, and \ref{thm:B}. 

\noindent \textbf{Acknowledgements.} The second author would like to thank the University of Fribourg for hosting her, during which an important part of this work was done. She also thanks Francisco Arana-Herrera for helpful discussions.

\section{Preliminaries and technical lemmas} 
\label{sec:prelims}
In this section, we establish some useful facts about hyperbolic geometry and surface amalgams which will be useful for subsequent sections. 
\subsection{Hyperbolic geometry}
Recall that $\mathbb{H}^2$, \textit{the hyperbolic plane}, may be described using the \textit{Poincar\'e disk model}. In other words, $\mathbb{H}^2$ is the unit disk $\{(x, y) \in \mathbb{R}^2: x^2 + y^2 < 1\}$ equipped with the metric described by $ds^2 = \dfrac{4(dx^2 + dy^2)}{(1 - (x^2 + y^2))^2}$. We will make abundant use of the following fact.

\begin{lemma}
The area of a hyperbolic disk of radius $r$ is $4\pi\sinh^2(r/2) = 2\pi(\cosh(r)-1)$.
\label{lemma:disk}
 \end{lemma} 

We state some facts about $\mathbb{H}^2$ which will prove useful in the arguments found in \Cref{sec:counting}. We begin with the following: 
\begin{lemma} Let $S$ and $\gamma$ be from before. Let $D$ be \emph{any} embedded closed disk of radius less than or equal to $\frac{\log(3)}{2}$ on $S$. Then the complementary regions of $\gamma \cap D$ are either quadrilaterals or half-disks. 
\label{lem:ln3/2}
\end{lemma}

\begin{figure}[H]\label{fig:nohexagons}   \centering
   \begin{tikzpicture}
    \node at (0,0)    {\includegraphics[width=0.8\linewidth]{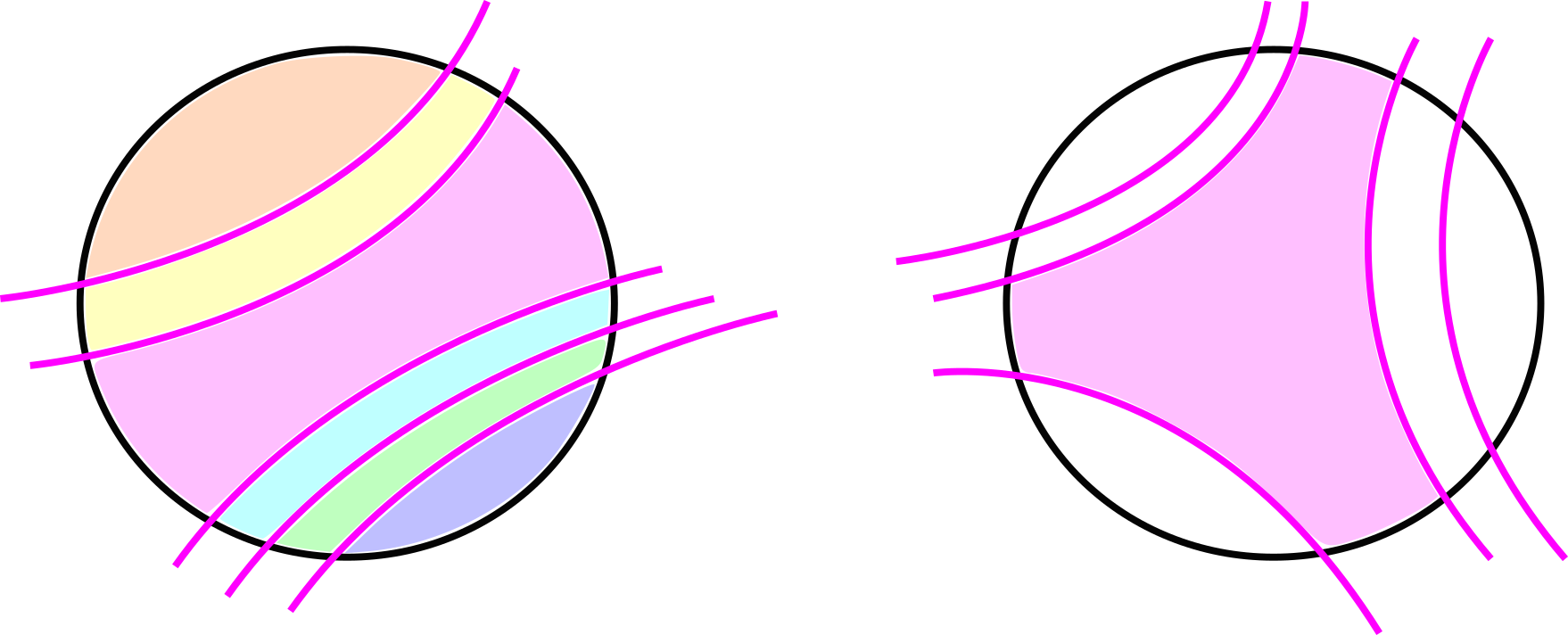}};
    \node at (-3.4, -3) {Possible};
    \node at (4, -3) {Not possible};
    \end{tikzpicture}
    \caption{The complementary regions of $\gamma \cap D$ must either be quadrilaterals or half disks if $D$ has radius less than or equal to $\frac{\log(3)}{2}$.}\end{figure}

\begin{proof}
If we lift the disk to the universal cover, and take the full lifts of the extended geodesic arcs passing through the interior of the disk, we obtain a collection of disjoint geodesics that pass through a disk of radius $\frac{\log 3}{2}$. 

Now suppose by contradiction that one of the complementary regions contained in the disk is larger than a quadrilateral. This means that there is a point that is in a region of the hyperbolic plane delimited by three (or more) disjoint geodesics, and at distance strictly less than $\frac{\log 3}{2}$ from these geodesics. By removing geodesics if necessary, we can assume that there are exactly three. If they do not form an ideal triangle already, we can certainly find an ideal triangle contained in the region to which the point belongs, and such that the distance to the point is again strictly less than $\frac{\log 3}{2}$. Now we can conclude because in an ideal triangle, there are no points at distance strictly less than $\frac{\log 3}{2}$ from all three sides. In fact, the value $\frac{\log 3}{2}$ is exactly the radius of the unique embedded disk tangent to all three sides of an ideal triangle.
\end{proof}

The following lemma establishes a lower bound on the angle measures of triangles inscribed in hyperbolic disks. 

\begin{lemma}[c.f. Lemma A.1, \cite{bp}] \label{lem:bigangles}
Consider, for given $r > 0$, a hyperbolic triangle with sides of lengths $\geq r$ inscribed in a circle of radius $\geq r$. Then all angles are bounded from below by $\varphi_{r}$, where    
$$\cot\bigg(\frac{\varphi_{r}}{2}\bigg) = \cosh(r)\cdot\bigg\{\sqrt{1 + 2\cosh(r)} + \sqrt{2 + 2\cosh(r)}\bigg\}.$$
\end{lemma}

\Cref{lem:bigangles} yields the following corollary, which will be used in the next section: 

\begin{corollary}\label{cor:neighborsbound} Given a minimal covering of $\mathbb{H}^2$ by disks, each disk has at most $\frac{2\pi}{\varphi_{r}}$ neighbors. 
    
\end{corollary}

\begin{proof} Fix a disk $D_1$ and consider $D_2$ and $D_3$, an adjacent pair of its neighbors. Consider the triangle contained in $D_1 \cup D_2 \cup D_3$ with vertices that are the centers of the three disks. Each of these sides has length at least $r$, the radius of the disks. Then there exists some $\varphi_r$ from \Cref{lem:bigangles} which is a lower bound on the angle of each triangle. Thus, since the total angle around the center of $D_1$ is $2\pi$, there are at most $\frac{2\pi}{\varphi_r}$ triangles around the center of $D_1$. See \Cref{fig:bigangles}. 
\end{proof}

\begin{figure}[H]
    \centering    
    \begin{tikzpicture}
    \node at (0,0) {\includegraphics[width=0.4\linewidth]{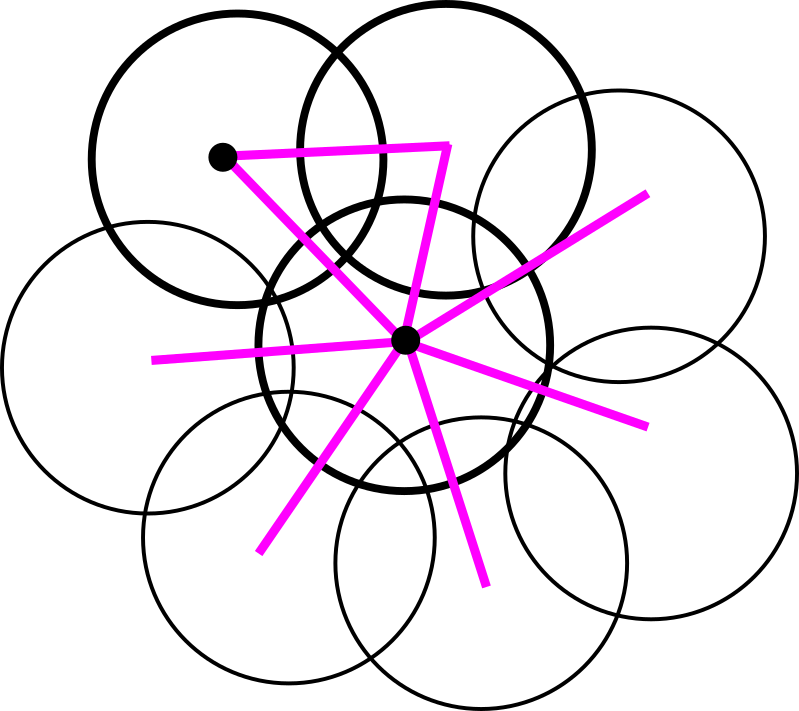}};
    \node at (-0.5, 0.3) [black]{$D_1$};
    \node at (-2.8,2.5) [black]{$D_2$};
    \node at (1, 2.9) [black]{$D_3$}; 
    \node at (-1.1, 2) [Magenta]{$\geq r$};
    \end{tikzpicture}
    \caption{}
    \label{fig:bigangles}
\end{figure}

\noindent \textbf{Fermi coordinates.} For ease of calculation, we now introduce \textit{Fermi coordinates}, which are a coordinate system for $(\mathbb{H}^2, \rho)$ defined with respect to a base bi-infinite geodesic $\widetilde{\gamma} \subset \mathbb{H}^2$. For the convenience of the reader, we briefly describe the coordinate system (see also page 4 of \cite{buser}). Note that $\widetilde{\gamma}$ separates $\mathbb{H}^2$ into the left and right sides. Given some point $p \in \mathbb{H}^2$, we define $t$ to be the point on $\widetilde{\gamma}$ where $\big(p, \widetilde{\gamma}(t)\big)$ is a segment perpendicular to the image of $\widetilde{\gamma}$. The other coordinate, $\rho$, is the signed distance from $p$ to $\widetilde{\gamma}$, where dist$(p, \widetilde{\gamma})$ is negative (resp. positive) if $p$ is on the left (resp. left) side of $\widetilde{\gamma}$. With respect to $(\rho, t)$, the usual hyperbolic metric on $\mathbb{H}^2$ can be expressed as 

\begin{equation}\label{eqn:fermi} ds^2 = d\rho^2 + \cosh^2(\rho) d t^2.\end{equation} 

\subsection{Surface amalgams}

We now provide the precise definition of a hyperbolic surface amalgam. Such objects were originally defined in \cite{buyalo} where they are called hyperbolic graph surfaces, and in \cite{lafont}, where they are called ``2-dimensional P-manifolds." 

\begin{definition}[Hyperbolic surface amalgam, c.f. Definition 2.3 of \cite{lafont}]\label{pmnfld} A compact metric space $X$ is a \textit{hyperbolic surface amalgam} if there exists a closed subset $Y \subset X$ (the \textit{gluing curves} of $X$) that satisfies the following: 

\begin{enumerate}
    \item Each connected component of $Y$ is homeomorphic to $S^1$;
    \item The closure of each connected component of $X - Y$ is homeomorphic to a compact hyperbolic surface with boundary, and the homeomorphism takes the component of $X - Y$ to the interior of a surface with boundary. We will call each $\overline{X - Y}$ a \textit{chamber} in $X$; 
    \item There exists a hyperbolic metric on each chamber which coincides with the original metric. 
\end{enumerate}
\end{definition}

In this paper, we also require the surface amalgams to be {\it proper}\footnote{Note that Lafont \cite{lafont} uses the terminology "thick", but since we will need to use thick-thin decompositions of the surface amalgams, we use ``proper".}, which means that they satisfy the following standing assumptions:

\begin{assumption} \label{assumption:simple} The surface amalgam $X$ is simple; that is, $Y$ forms a totally geodesic subspace of $X$ consisting of disjoint simple closed curves.
\end{assumption}

\begin{assumption} \label{assumption:thick}
Each connected component of $Y$ (gluing curve) is attached to at least three distinct boundary components of (not necessarily distinct) chambers. 
\end{assumption}

\Cref{assumption:simple} ensures our surface amalgam is a locally CAT($-1$) space. \Cref{assumption:thick} rules out surface amalgams with nonempty boundary as well as closed hyperbolic surfaces. In particular, note that the results of this paper do not hold for closed hyperbolic surfaces, which \emph{always} have topological and volume entropy equal to $1$. 

\begin{figure}[h]
    \centering
    \includegraphics[width=\textwidth]{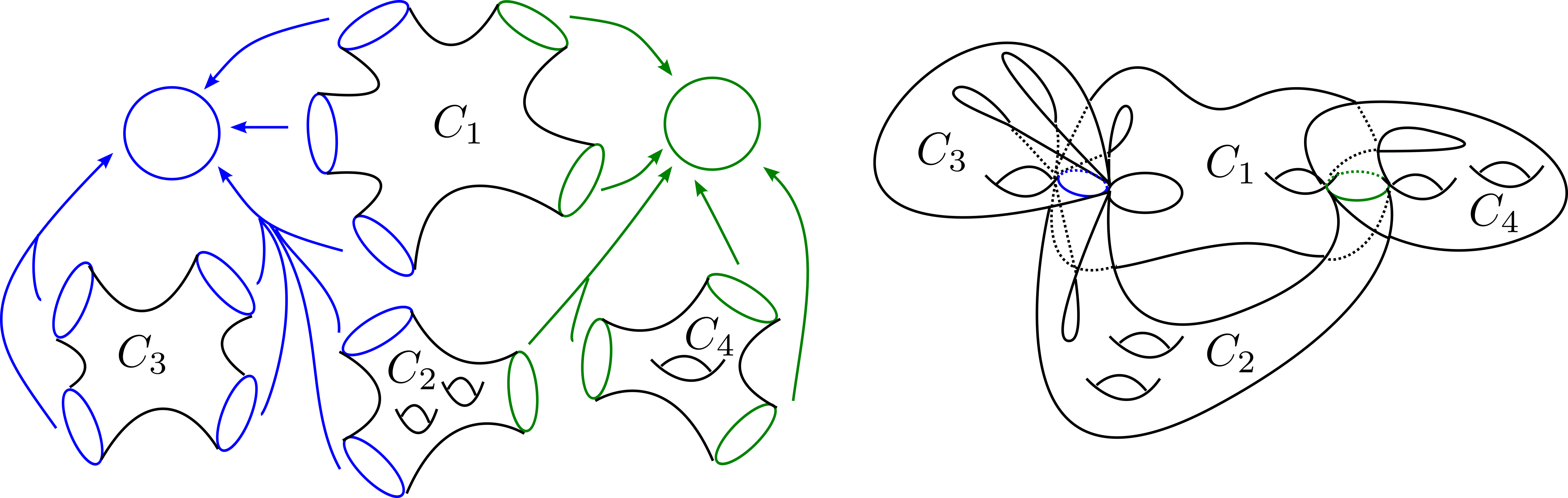}
    \caption{An example of a proper hyperbolic surface amalgam with four chambers.}
    \label{fig:pmnfld}
\end{figure}
\noindent\textbf{Collar lemmas and strip formulas.} We define the \textit{$r$-neighborhood of $\gamma$}, where $\gamma$ is a closed geodesic, to be the set of points $$N_{\gamma}(r) = \{x \in X: \text{dist}(x, \gamma) \leq r\}.$$ The following lemma is a straightforward generalization of the formula for the area of a strip around a simple closed geodesic on a hyperbolic surface. 

\begin{lemma}\label{lemma:strip} Let $r > 0$ and let $\gamma$ be a simple closed geodesic and suppose $N_{\gamma}(r)$ $\gamma$ is a union of embedded annuli. Then $$\Area\big(N_{r}(\gamma)\big) = n(\gamma)\ell(\gamma)\sinh(r),$$ where $n(\gamma)$ is the number of boundary components of chambers attached to $\gamma$ if $\gamma$ is a gluing curve and $2$ otherwise. 
\end{lemma}

\begin{proof}
This is a straightforward calculation using Fermi coordinates defined with respect to a connected component of the preimage of $\gamma$ under the covering map. Since the piecewise hyperbolic metric on $X$ is the usual hyperbolic metric when restricted to a single annulus $A$ embedded in $\mathscr{C}(X)$, we can use the coordinates given in \Cref{eqn:fermi} to calculate the area of the annulus: 

$$\int_{\rho = -r}^{r} \int_{t = 0}^{\ell(\gamma)} \cosh(\rho)dtd\rho = 2\sinh(r)\ell(\gamma).$$

If we change the bounds of integration of the outside integral to $\rho = 0$ instead of $\rho = -r$ or $\rho = r$, we see that the integral would evaluate to $\sinh(r)\ell(\gamma)$. This implies that $\gamma$ divides $A$ into two annuli with the same area. Thus, the theorem follows. 
\end{proof}

We define the \textit{collar} $\mathscr{C}(\gamma)$ around a closed geodesic $\gamma$ in a surface amalgam analogously with how they are defined for surfaces. Namely, $\mathscr{C}(\gamma)$ is a $r$-neighborhood of width $$w(\gamma) = \arcsinh\bigg(\frac{1}{\sinh(\ell(\gamma)/2)}\bigg).$$
Recall that collars in hyperbolic surfaces are annuli due to the classical Collar Lemma; see \cite{keen} and Theorems 3.1.8 and 4.1.1 in \cite{buser}. In the case of surface amalgams, the topologies of the collars can be much more complicated; they are not necessarily homeomorphic to annuli, or even unions of annuli, if the core curves intersect gluing geodesics (see \Cref{fig:collars}). As a result, a more nuanced version of the collar lemma, \Cref{lem:collar}, is required.  

\begin{figure}[H]
    \centering
    \includegraphics[width=\linewidth]{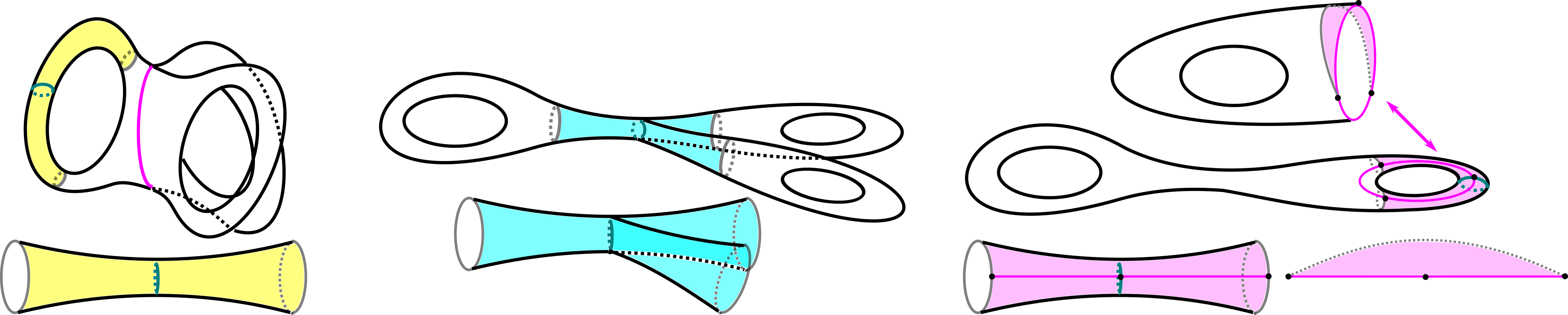}
    \caption{There are more topological types for collars of short curves in surface amalgams.}
    \label{fig:collars}
\end{figure}


\begin{lemma}[Collar lemma for proper hyperbolic surface amalgams] \label{lem:collar} Let $X$ be a proper hyperbolic surface amalgam. Then any collection of simple closed geodesics $\gamma_1, \gamma_2, ..., \gamma_m$  which are pairwise disjoint and do not intersect gluing curves unless they are themselves gluing curves satisfy the following: 

\begin{enumerate}
\item The collars $\mathscr{C}(\gamma_i)$ of widths $w(\gamma_i)$ given by the formulas $$\mathscr{C}(\gamma_i) = \{p \in S: \text{dist}(p, \gamma_i) \leq w(\gamma_i)\} \text{ where } w(\gamma_i) = \arcsinh\big(1/\sinh(\ell(\gamma_i) / 2)\big)$$ are pairwise disjoint for $i = 1, 2, ..., 3g - 3$. 
\item Each $\mathscr{C}(\gamma_i)$ is a finite union of $n(\gamma_i)$ cylinders each isometric to $[0, w(\gamma_i)] \times S^1$ with the Riemannian metric $ds^2 = d\rho^2 + \ell^2(\gamma_i)\cosh^2(\rho)dt^2$, where $n(\gamma_i)$ is the number of boundary components of chambers attached to $\gamma_i$ if $\gamma_i$ is a gluing curve and $2$ otherwise. The cylinders are all identified along $\gamma_i$, their common boundary component. 
\end{enumerate} 
\end{lemma}

\begin{proof} First, suppose that $\gamma_i$ is disjoint from any $\gamma_j$ where $i \neq j$, and $\gamma_i$ does not intersect any gluing geodesics. Then $\gamma_i$ is entirely contained in some chamber of $C \subset X$, as it is disjoint from all gluing geodesics and thus cannot span multiple chambers. We can then apply the classical Collar Lemma to $C$, which is a compact, hyperbolic surface, and conclude $\mathscr{C}(\gamma_i)$ is an annular neighborhood of $\gamma_i$.

Suppose $\gamma_i$ is a gluing geodesic attached to some collection of chambers $\{C_j\}_{j = 1}^{n(\gamma_i)}$. By the classical collar lemma, see, e.g., Theorem 4.3.2 of \cite{buser}, on each $C_j$, there is an annular neighborhood of $\gamma_i$ which is a half-collar of width $\arcsinh\left(1/\sinh\left(\frac{\ell(\gamma_i)}{2}\right)\right)$. Then $\mathscr{C}(\gamma_i)$ consists of $n(\gamma_i)$ copies of these half-collars glued together along $\gamma_i$, their shared boundary component. 

Since the $\gamma_i$ satisfying the hypotheses in the statement of the corollary are either boundary components of compact hyperbolic surfaces or disjoint simple closed geodesics in compact hyperbolic surfaces, we can use the standard collar lemma to conclude their collars are disjoint. 
\end{proof} 

Corollary 4.1.2 of \cite{buser} implies that short simple closed geodesics in compact hyperbolic surfaces cannot intersect (transversely). This is not true in the case of surface amalgams; it is possible for two short curves that intersect some gluing geodesic to also intersect each other, for instance (see \Cref{fig:shortintersecting}). The following theorem states, however, that in all other cases, two short curves cannot intersect. 
\begin{lemma}\label{lemma:transverse} Let $\gamma$ and $\delta$ be closed geodesics on $X$ that intersect each other at a finite number of points outside the set of gluing curves $\Gamma$, and suppose $\gamma$ is simple. Then: 

$$\sinh\bigg(\frac{1}{2}\ell(\gamma)\bigg)\sinh\bigg(\frac{1}{2}\ell(\delta)\bigg) \geq 1.$$

\end{lemma}

\begin{proof} Let $\widetilde{\gamma}$ and $\widetilde{\delta}$ be lifts of $\gamma$ and $\delta$ respectively which intersect at a single point in $\widetilde{X}$. Consider $N_{\epsilon}(\widetilde{\gamma} \cap \widetilde{\delta})$, where $\epsilon \geq \arcsinh\big(\frac{1}{\sinh(\ell(\gamma)/2)}\big)$. Let $A \cong \mathbb{H}^2$ be an apartment in $\widetilde{X}$ which contains $\widetilde{\gamma}$, which exists by Lemma 2.9 of \cite{wu}. While $\alpha_{\widetilde{\delta}} = \widetilde{\delta} \cap N_{\epsilon}(\widetilde{\gamma} \cap \widetilde{\delta})$ may not necessarily be contained in $A$, there is a natural isometric projection of $N_{\epsilon}(\widetilde{\gamma} \cap \widetilde{\delta})$ onto an open disk in $A$. Since $\gamma$ and $\delta$ do not intersect on a gluing geodesic and do not agree on any geodesic segments, their images under the projection map will be transverse. Thus, the image of $\alpha_{\widetilde{\delta}}$ under this projection is an arc that connects the two boundary components of $A \cap \widetilde{\mathscr{C}}(\widetilde{\gamma})$. Since all apartments share the same metric, the projection preserves lengths; hence, $\ell_{X}(\delta) \geq \ell_{\widetilde{X}}(\alpha_{\widetilde{\delta}}) \geq 2\arcsinh\big(\frac{1}{\sinh(\ell(\gamma)/2)}\big)$. Since hyperbolic sine is an increasing function, 

$$\text{sinh}\bigg(\frac{1}{2}\ell(\delta)\bigg) \geq \frac{1}{\text{sinh}(\frac{1}{2}\ell(\gamma))}.$$

The inequality thus follows. Note that equality occurs if we replace $\ell(\gamma)$ and $\ell(\delta)$ with $2\arcsinh(1)$.
\end{proof}

\begin{figure}[H]
    \centering
    \includegraphics[width=0.6\linewidth]{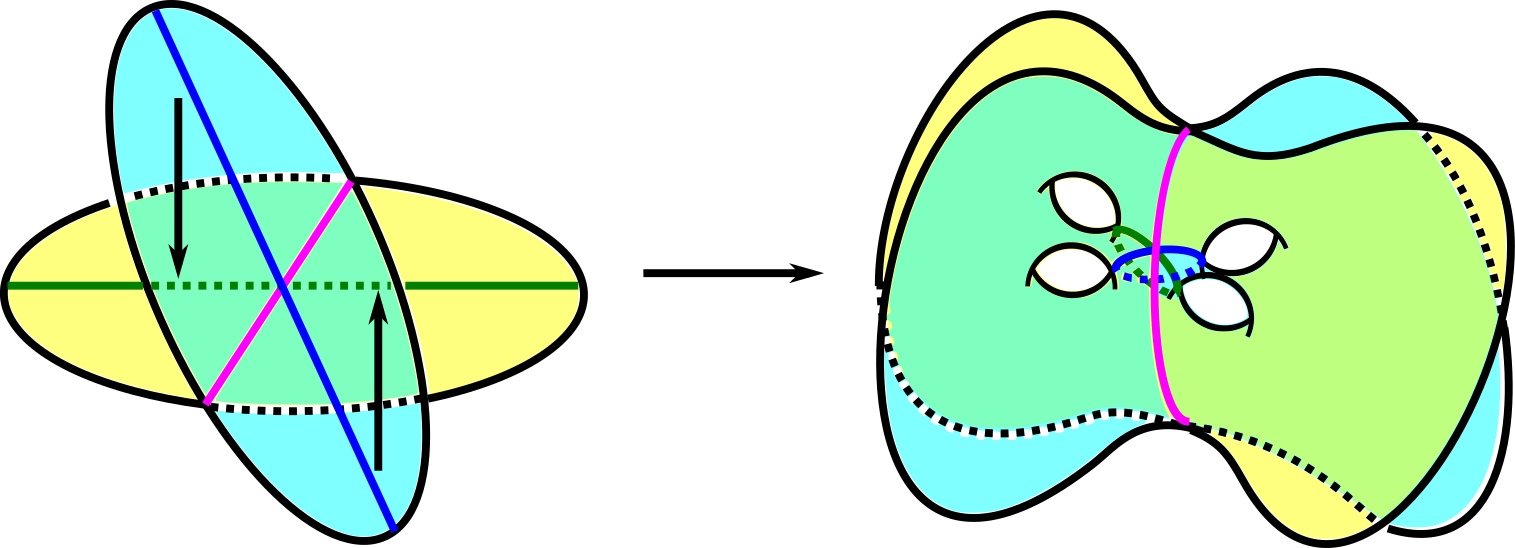}
    \caption{If two short geodesics intersect at a gluing curve, the images of their lifts may project to the same bi-infinite closed geodesic, so the proof for \Cref{lemma:transverse} will not work.}
    \label{fig:shortintersecting}
\end{figure}

\noindent \textbf{Thick and thin decompositions.} \label{def:thickandthin} Recall that a consequence of the collar lemma for hyperbolic surfaces is the existence of \textit{thick and thin decompositions}. In this subsection, we define thick and thin decompositions for surface amalgams. 

\begin{definition}[Injectivity radius for surface amalgams] \label{def:injrad} We define the \textit{injectivity radius of $X$ at $p$}, $r_p(X)$, to be the supremum over $r > 0$ where $B(p, r)$ is homeomorphic to either a disk or the product of a tree with an interval (see \Cref{fig:tree}). 
\end{definition} 

Equivalently, the injectivity radius is twice the length of the shortest simple geodesic loop based at $p$. Note that the definition of $r_p(S)$ for a surface $S$ coincides with the definition above. 

\begin{definition}[Thick and thin decomposition of $X$] The \textit{thin part of $X$} is the set $\{p \in X: r_p(X) < \frac{\arcsinh(1)}{2} \}$. The complement of $X$ is the \textit{thick part} of $X$. 
\end{definition}

We next show that, as in the case of closed hyperbolic surfaces (c.f. Theorem 4.1.6 of \cite{buser}), points with``large" injectivity radius lie in the complement of the union of collars of short curves, which is made precise by the following.

\begin{lemma} \label{lem:thick} Let $X$ be a proper hyperbolic surface amalgam. Let $\{\gamma_1, \gamma_2, ..., \gamma_k\}$ be the collection of simple closed geodesics in $X$ that have length $< 2\arcsinh(1)$. If $p \in X \setminus \bigg\{\bigcup\limits_{\gamma_i} \mathscr{C}(\gamma_i)\bigg\}$, then $r_p(X) \geq \arcsinh(1)$. 
\end{lemma}

\begin{proof} Suppose that $r_p(X) < \arcsinh(1)$. Then by definition, there is some simple geodesic loop $\mu_p$ based at $p$ of length $2r_p(X) \leq 2\arcsinh(1)$. Since $X$ is locally CAT($-1$), $\mu_p$ is freely homotopic to some closed geodesic $\beta$. By Lemma 2.9 of \cite{wu}, $\beta$ lifts to a bi-infinite geodesic $\widetilde{\beta}$ contained in some apartment $A \subset \widetilde{X}$ which is isometric to a hyperbolic plane (but in general may not project to a closed surface under the covering map). While there may not be a lift $\mu_p$, $\widetilde{\mu}_p$, which lies in $A$, there is a natural projection $\pi: \widetilde{X} \rightarrow A$ which maps some $\widetilde{\mu_p}$ onto $A$. The projection is distance-preserving since $X$ is a hyperbolic surface amalgam. Since $\pi(\widetilde{\mu_p})$ and $\widetilde{\beta}$ lie in some hyperbolic plane, we can apply the arguments from the proof of Theorem 4.1.6 in \cite{buser} to conclude that $\ell(\beta) < 2\arcsinh(1)$, so $\beta$ is one of the core curves of the collars, and $p \in \mathscr{C}(\beta)$.  
\end{proof}

\Cref{lem:thick} provides a convenient characterization of the thin and thick parts of $X$. We can think of the thin part of a hyperbolic surface amalgam as the set of collars $\{\mathscr{C}(\beta_i)\}$, where $\{\beta_i\}$ is the set of closed geodesics in $X$ with length less than $2\arcsinh(1)$. The complement of the collar neighborhoods $X \setminus \big\{\bigcup\limits_{i} \mathscr{C}(\beta_i)\big\}$ is the thick part of $X$.

We remark that unlike in the case of hyperbolic surfaces, the collars $\mathscr{C}(\beta_i)$, and even the simple closed curves $\beta_i$ are not necessarily disjoint. 

\subsection{Entropies and the count of geodesics}
\label{subsec:entropy}

In the final part of this preliminary section, using results proven in the literature, we establish a relationship between entropy and counting (see \Cref{thm:entropy}). We first introduce the notion of topological entropy of the geodesic flow map on the \textit{generalized unit tangent bundle} $SX$ of $X$, or the space of unit-speed geodesics in $X$ (see Section 3.1 of \cite{wu} for more details).

Let $\mathcal{G}(X)$ be the set of geodesics of $X$, and let $\varphi: \mathbb{R} \rightarrow X$ be the geodesic flow map on $SX$. Define 
\begin{align*} d_{\mathcal{G}(X)}(\gamma_1, \gamma_2) &= \int_{\mathbb{R}} d_X(\gamma_1(t), \gamma_2(t))\frac{e^{-\lvert t \rvert}}{2}dt \text{, where }
d_T(x, y) = \max\limits_{0 \leq t \leq T} d_X(\varphi_t(x), \varphi_t(y)).
\end{align*}

We say that a set $S \in \mathcal{G}(X)$ is \textit{$(T, \epsilon)$-separated} if for all $\gamma_1, \gamma_2 \in S$, $d_T(\gamma_1, \gamma_1) \geq \epsilon.$ We say $S \subset \mathcal{G}(X)$ is \textit{$(T, \epsilon)$-spanning} if for any $\gamma \in \mathcal{G}(X)$, $d_T(\gamma, \gamma') \leq \epsilon$ for some $\gamma' \in S$. Let $N(\varphi_, \epsilon, T)$ be the \textit{maximum} cardinality of a $(T, \epsilon)$-separated set, and $S(\varphi, \epsilon, T)$ be the \textit{minimum} cardinality of a $(T, \epsilon)$-spanning set. 

\begin{definition}[Topological entropy]
There are several equivalent definitions.  

$$h_{\text{top}}(\varphi) = \lim\limits_{\epsilon \rightarrow 0}\lim\limits_{T \rightarrow \infty} \sup \frac{\log\big(N(\varphi, \epsilon, T)\big)}{T} = \lim\limits_{\epsilon \rightarrow 0}\lim\limits_{T \rightarrow \infty} \sup \frac{\log\big(S(\varphi, \epsilon, T)\big)}{T}.$$
\end{definition}

The following is another notion of entropy which is widely studied in the literature. Note that in this section, $(X,g)$ will be a proper surface amalgam where $g$ is the metric.

\begin{definition}[Volume Entropy]
    The \textit{volume entropy} of a proper hyperbolic surface amalgam $(X, g)$ is the limit $$h_{\text{vol}}(X, g) = \lim\limits_{r \rightarrow \infty} \frac{\log\big(\text{vol}_g\big(B(x, r)\big)\big)}{r}$$ where $x \in \widetilde{X}$ is a basepoint in the universal cover $\widetilde{X}$ of $X$, and $B(x, r)$ is the ball of radius $r$ centered at $x$ whose volume depends on $g$.
\end{definition}

The following theorem suggests that the volume and topological entropies are in fact equal in the setting of proper hyperbolic surface amalgams. 

\begin{theorem}[c.f. Main Theorem of \cite{leuzinger}.] Let $(X, g)$ be a proper hyperbolic surface amalgam. Then $$h_{\text{top}}(\varphi) = h_{\text{vol}}(X, g),$$ where $h_{\text{top}}(\varphi)$ is the topological entropy of the geodesic flow $\varphi_t: \mathbb{R} \rightarrow X$ on the space of geodesics on $X$, and $h_{\text{vol}}(X, g)$ is the volume entropy of $(X, g)$.
\end{theorem}

\begin{proof} This is a direct consequence of the main theorem of \cite{leuzinger}. We can check all the conditions from Leuzinger's theorem. The main theorem first requires $(\widetilde{X}, \widetilde{g})$ to be geodesically complete, locally uniquely geodesic, and geodesic. Indeed, in $(\widetilde{X}, \widetilde{g})$, each geodesic extends indefinitely, and there is a unique length-minimizing geodesic path joining any two points in $(\widetilde{X}, \widetilde{g})$. Furthermore, $\pi_1(X)$ acts properly discontinuously and cocompactly by isometries on $(\widetilde{X}, \widetilde{g})$ and $\pi_1(X)$ preserves the natural volume measure on $(X, g)$. Additionally, since $(X, g)$ is a complete CAT($0$) (Hadamard) space, $(X, g)$ satisfies the so-called Property (C) required in the statement of the main theorem of \cite{leuzinger}. 

The only assumption left to check is that $(X, g)$ satisfies what Leuzinger calls Property (U); that is, we need to show there exist $0 < \delta_0 \leq \infty$ such that for all $0 < \delta < \delta_0$, there exist positive constants $C_i(\delta)$, where $i = 1, 2$, such that 
$$0 < C_1(\delta) = \inf\limits_{x \in \widetilde{X}} \text{vol}_g\big(B(x, \delta)\big) \leq \sup\limits_{x \in \widetilde{X}} \text{vol}_g\big(B(x, \delta)\big) = C_2(\delta).$$

To see this, note that balls in $X$ are either disks, which have the smallest area, or are homeomorphic to the product of trees with the interval. Let $C$ be any chamber in $X$. There exists some disk of radius $r_C$ which is completely contained in the interior of $C$. Set $\delta_0 = r_C$. Then for $0 < \delta < \delta_0$, $\widetilde{X}$ contains infinitely many balls of area $$4\pi\sinh^2(\delta/2) = \inf_{x \in \widetilde{X}}\text{vol}_g\big(B(x, \delta)\big).$$ On the other hand, if we choose some $\widetilde{y} \in \widetilde{X}$ which lies on a lift of a gluing geodesic, the ball $B(y, \delta)$ contains a disk of radius $\delta$ centered at $y$, so $$\sup\limits_{x \in \widetilde{X}} \text{vol}_g\big(B(x, \delta)\big) \geq \text{vol}_g\big(B(y, \delta)\big) > 4\pi\sinh^2(\delta/2).$$
Then by \cite{leuzinger}, $h_{\text{top}}(\varphi) = h_{\text{vol}}(X, g)$. 
\end{proof}

Recall from before that $\#\G_X(L)$ is the number of closed geodesics in $(X, g)$ of length less than or equal to $L$. The relationship between $\#\G_X(L)$ and the topological entropy of the geodesic flow map on $SX$ is given by the following special case of a general theorem due to Ricks:  

\begin{theorem}[c.f. Theorem A of \cite{rickscounting}] \label{thm:entropy} Let $X$ be a proper hyperbolic surface amalgam. Then as $L \rightarrow \infty$, $$ \#\G_X(L) \sim \dfrac{e^{hL}}{hL}$$ where $h = h_{\text{top}}(\varphi)$ is the topological entropy of the geodesic flow on the generalized unit tangent bundle $SX$. 
\end{theorem}

\begin{remark} As remarked in \cite{rickscounting}, the exponent $h$ is actually equal to the \textit{critical exponent} $\delta_{\pi_1(X)}$ of the Poincar\'e series of $\pi_1(X)$. However, he shows in an earlier paper that if $\widetilde{X}$ is a proper and geodesically complete CAT(0) metric space, then $\delta_{\pi_1(X)}$ is equal to the topological entropy of the geodesic flow map on $\pi_1(X) / SX$ (see Theorem A in \cite{ricksentropies}).
\end{remark}

Although we assume properness, the theorem is still true if $X$ is a closed hyperbolic surface. In this classical setting, $h = 1$, which yields the usual count of $\lvert \#\G_X(L)\rvert \sim e^L/L$ originally due to Huber (\cite{huber}, see also the appendix of \cite{es}).

As an immediate corollary to the above, we have the following result which will allow us to provide bounds on topological (or volume) entropy in what follows. We use the notation $\log$ for the natural logarithm. 

\begin{corollary}
The topological entropy $h$ of a proper hyperbolic surface amalgam $X$ satisfies
\[
\lim_{L\to \infty} \frac{\log\left(\#\G_X(L)\right)}{L}=h.
\]
\end{corollary}

\section{Counting Geodesics}
\label{sec:counting}
In this section, we will derive upper and lower bounds for $\#\G_X(L)$, where
$$\#\G_X(L) = \#\{\gamma : \gamma \text{ is a nontrivial unoriented closed geodesic and } \ell(\gamma) \leq L\}.$$ 

\noindent \textbf{Notation.} The following notation will be used throughout. Let $\Gamma$ be the multicurve consisting of gluing geodesics of $X$. Set $$B = B(\Gamma) := \sum\limits_{\gamma \in \Gamma} \ell(\gamma),$$ the sum of the lengths of all the gluing curves in $X$. Recall that in the introduction, we defined the quantity $$r_0 := \min \bigg\{\frac{\log(3)}{4}, \frac{\sys(X)}{4}\bigg\}.$$  

Observe that $r_0 < \frac{\arcsinh(1)}{2}.$ For convenience, we set $A := \Area(X).$

\color{black}

\subsection{Upper Bound}

We first find an upper bound for $\#\G_X(L)$. Cover $X$ by balls of radius $r_0$ in the following way: choose a maximal collection of disjoint balls of radius $r_0 / 2$, centered around points $\{p_1, ..., p_n\}$. Note that by \Cref{def:injrad}, our choice of $r_0$ ensures that balls of radius $2r_0$ are embedded, which will be useful later. 

First, we observe that the set of balls $\{B(p_i, r_0)\}_{i = 1}^{n}$ covers $X$. If not, then there exists some $p \in X$ such that $p \not\in B(p_i, r_0)$ for all $i = 1, 2, ..., n$, and hence $B(p,r_0/2)$ and $B(p_i,r_0/2)$ are disjoint for all $i$ and the collection of disjoint balls is not maximal. 

The proof will then proceed in the following steps: 

\begin{itemize}
\item \textit{Step 1:} Find an upper bound on $n$, the number of balls in the covering $\bigcup\limits_{i = 1}^n B(p_i, r_0)$;
\item \textit{Step 2:} Bound the area of each ball in $\bigcup\limits_{i = 1}^n B(p_i, r_0)$; 
\item \textit{Step 3:} Use the previous step to bound $\beta(L)$, the number of balls a closed geodesic $\eta$ of length at most $L$ passes through; 
\item \textit{Step 4:} Find an upper bound $\lambda$ for the number of neighbors adjacent to any given ball in $\bigcup\limits_{i = 1}^n B(p_i, r_0)$. An upper bound on $\#\G_X(L)$ will be given by $n \lambda^{\beta(L)}$. 
\end{itemize}

\noindent \textbf{Step 1: Bound on the the number of balls in the covering.} By \Cref{lemma:disk}, the area of a disk of radius $r_0/2$ in $\mathbb{H}^2$ is $4\pi \sinh^2(r_0/4)$. Since each ball $B(p_i,r_0/2)$ is embedded and $X$ is locally $\mathrm{CAT}(-1)$, its area is bounded below by the area of a hyperbolic disk of radius $r_0/2$. Thus, since $\sinh^2(x) > x^2$ when $x > 0$,  
$$\Area\big(B(p_i, r_0/2)\big) \geq 4\pi\sinh^2(r_0/4) > \frac{\pi r_0^2}{4}.$$

In summary, since the disks of radius $r_0 / 2$ are disjoint, the total number of balls that cover $X$ is at most 
\begin{equation} \label{eqn:numberofballsbound}
n\leq \frac{4A}{\pi r_0^2}.
\end{equation}

\noindent \textbf{Step 2: Bound on the area of each ball in the covering.} We first observe the following: 

\begin{lemma} Let $B(x, r)$ be an embedded ball in $X$, where $r < \log(3)/2$. Then $B(x, r)$ is homeomorphic to the product of a (finite) tree and $(0, 1)$.  \label{lem:ballhomeo}
\end{lemma} 

 \begin{proof} Note that due to \Cref{lem:ln3/2}, the complementary regions of $\Gamma \cap B(x, r)$ are either half disks or quadrilaterals. If we replace each half disk complementary region with a rectangle, we obtain a tree $T$ crossed with an interval. The interior vertices of $T$ correspond to gluing geodesics that intersect $B(x, r)$, while the leaves correspond to complementary regions which are half disks. See \Cref{fig:tree}. \end{proof} 

\begin{figure}[H]
    \centering    
    \begin{tikzpicture}
    \node at (0,0) {\includegraphics[width=\linewidth]{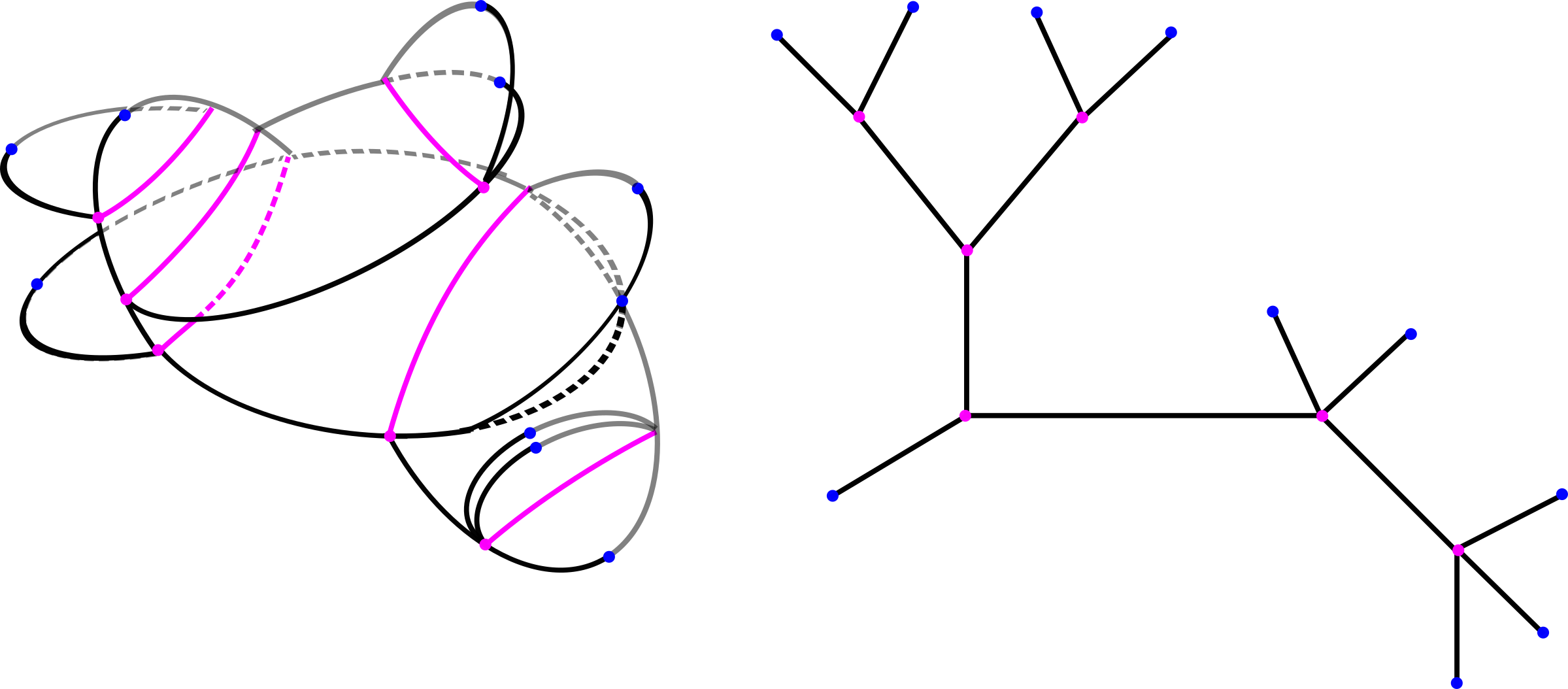}};
    \node at (-7.9,2) [blue]{$1$};
    \node at (-6.5,2.6) [blue]{$2$};  
    \node at (-6.7,1.6) [Magenta]{$3$}; 
    \node at (-6.5,0.8) [Magenta]{$4$};
    \node at (-3,2) [Magenta]{$5$};
    \node at (-3.1,3.6) [blue]{$6$};
    \node at (-2.9,2.9) [blue]{$7$};
    \node at (-6.2,-.35) [Magenta]{$8$};
    \node at (-7.5,0.8) [blue]{$9$};    
    \node at (-3.4,-.6) [Magenta]{$10$};
    \node at (-1.4,1.9) [blue]{$11$};
    \node at (-1.1,0) [blue]{$12$};  
    \node at (-3.1,-2.2) [Magenta]{$13$}; 
    \node at (-3,-1) [blue]{$14$};      
    \node at (-2.1,-1.1) [blue]{$15$};  
    \node at (-2,-2.5) [blue]{$16$};
    \node at (-0.3,3) [blue]{$1$};
    \node at (0.9,3.3) [blue]{$2$};  
    \node at (0.5,2) [Magenta]{$3$}; 
    \node at (1.4,0.9) [Magenta]{$4$};
    \node at (3.1,2) [Magenta]{$5$};
    \node at (2.8,3.3) [blue]{$6$};
    \node at (4.1,3) [blue]{$7$};
    \node at (1.8,-1) [Magenta]{$8$};
    \node at (0.8,-1.7) [blue]{$9$};    
    \node at (5,-1) [Magenta]{$10$};
    \node at (4.5,0) [blue]{$11$};
    \node at (5.7,0.2) [blue]{$12$};  
    \node at (6.1,-2.1) [Magenta]{$13$}; 
    \node at (6.3,-3.1) [blue]{$14$};      
    \node at (7.1,-2.9) [blue]{$15$};  
    \node at (7.6,-1.9) [blue]{$16$};    
    \node at (-4.5,-3) {$B(x, r)$};
    \node at (3.5,-3) {$T$};
    \end{tikzpicture}     
     \caption{The (open) ball $B(x, r)$ is homeomorphic to $T \times (0, 1)$, where $T$ is a tree.}
     \label{fig:tree}
\end{figure}
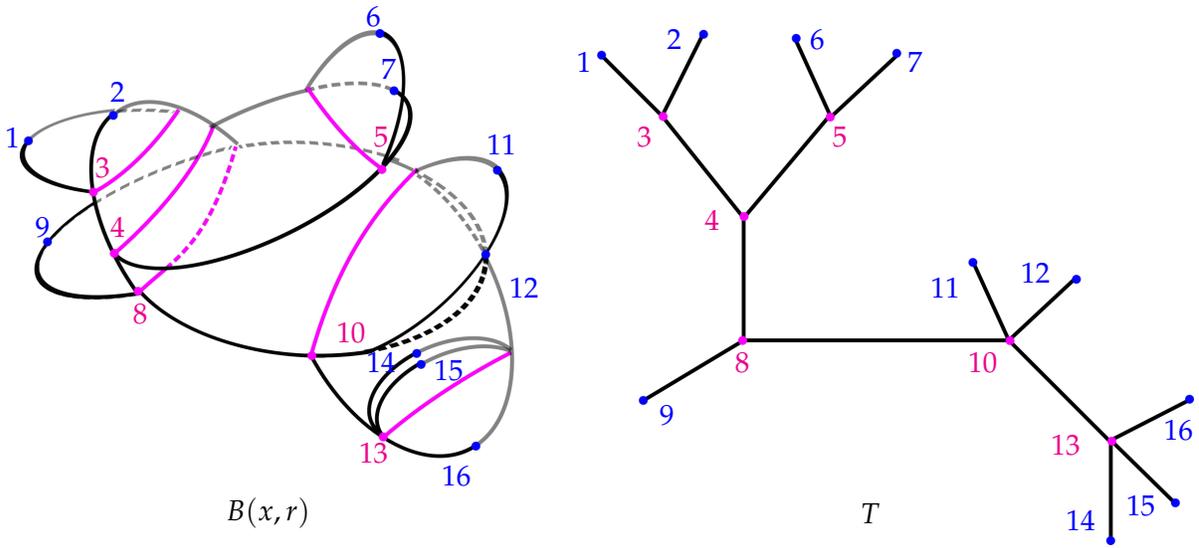 

Suppose a ball $B(p_i, r_0)$ is homeomorphic to $T_i \times (0, 1)$, and is not homeomorphic to a disk. In the subsequent paragraphs, we will bound the number of leaves of $T_i$.  


First, we bound the number of interior vertices of $T_i$, which is also the number of connected components of $\Gamma \cap B(p_i, r_0)$ ($i = 1, 2, ..., n$). Consider a connected component $s \subset \Gamma \cap B(p_i, r_0)$. Let $\tilde{s}$ be the connected component of $\Gamma \cap B(p_i, 2r_0)$ containing $s$. Then $\tilde{s}$ has length at least $2r_0$ (see \Cref{fig:doubled}). Since distinct components $\tilde{s}$ are disjoint subarcs of $\Gamma$ and $\ell(\Gamma)=B$, it follows that the number of connected components of $\Gamma \cap B(p_i, r_0)$ is at most $B/2r_0$.

\begin{figure}
    \centering
    \begin{tikzpicture}
    \node at (0, 0){\includegraphics[width=0.3\linewidth]{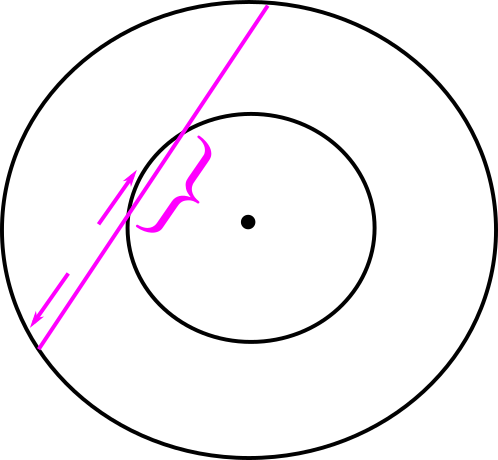}};
    \node at (-1.85, -0.1) [Magenta]{$\geq r_0$}; 
    \node at (-0.3, 0.2) [Magenta]{$s$};
    \node at (-0.5, 1.6) [Magenta]{$\widetilde{s}$}; 
    \node at (0.3, 0) {$p_i$}; 
    \node at (0.9, 1.3) {$B(p_i, r_0)$};
    \node at (3, 1) {$B(p_i, 2r_0)$}; 
    \end{tikzpicture}
    \caption{The segment $s$ can be arbitrarily short, but its extension $\tilde{s}$ has length at least $2r_0$.}
    \label{fig:doubled}
\end{figure}
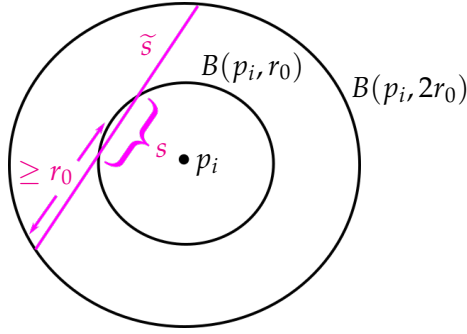

The valence of each interior vertex of $T_i$ is given by the number of chambers adjacent to the segment $s$ of the gluing geodesic corresponding to the vertex. Recall that the area of a hyperbolic pair of pants is $2\pi$, which is also a lower bound on the area of each chamber of $X$. 

Although chambers need not be pairs of pants, we may fix a pants decomposition of each chamber. 
This decomposes $X$ into hyperbolic pairs of pants whose boundary curves are either gluing geodesics in $\Gamma$ or internal curves lying entirely inside a chamber. 
Since each pair of pants has area $2\pi$, the total number of pairs of pants in $X$ is at most $\mathrm{Area}(X)/(2\pi)$. 
Moreover, a pair of pants has three boundary components, so it can be adjacent to a given gluing geodesic along at most three boundary components. 
It follows that the valence of each interior vertex of $T_i$ is bounded above by $\frac{3\mathrm{Area}(X)}{2\pi}$.

Then, using the bound of $B/2r_0$ on the number of interior vertices, by the Handshake Lemma we have
\begin{equation}\label{eqn:leaves}\#\{\text{leaves of } T_i\} \leq (B / 2r_0)\bigg(\frac{3A}{2\pi} - 2\bigg) + 2. \end{equation}  

We now use \Cref{eqn:leaves} to bound the area of each $B(p_i, r_0)$. Set a distinguished interior vertex of $T_i$, $x_0$, to be the root of $T_i$. Consider the collection of paths from $x_0$ to each leaf of $T_i$; each of these paths corresponds to a topological half-disk in a chamber in $B(p_i, r_0)$ with area bounded above by $4\pi\sinh^2(r_0/2)$. Therefore,
\begin{align}
\Area\big(B(p_i, r_0)\big) 
&\leq \big(\# \{\text{leaves of } T_i\}\big)\big(4\pi\sinh^2(r_0/2)\big) \notag\\
&\leq \bigg[(B / 2r_0)\bigg(\frac{3A}{2\pi} - 2\bigg) + 2\bigg]\big(4\pi\sinh^2(r_0/2)\big). 
\label{eqn:areaofball}
\end{align}

\noindent \textbf{Step 3: Bound the number of balls a long geodesic passes through.} 
Let $\eta$ be a closed geodesic of length $L$, and let $\widetilde{\eta}$ be one of its lifts to the universal cover $\widetilde{X}$. 
Consider the set of balls in $\widetilde{X}$ that intersect $\widetilde{\eta}$; these are lifts of (not necessarily pairwise distinct) balls in the collection $\{B(p_i, r_0)\}$. 
As discussed earlier, the balls $\{B(p_i, r_0 / 2)\}$ are disjoint, and the union of those balls $B(p_i,r_0/2)$ which intersect $\widetilde{\eta}$ is contained in the neighborhood of radius $\frac{3r_0}{2}$ around $\widetilde{\eta}$, namely $N(\widetilde{\eta}, \frac{3r_0}{2})$. 
Indeed, the center of each ball $B(p_i, r_0)$ is distance less than $r_0$ away from $\widetilde{\eta}$, so any point in $B(p_i, r_0 / 2)$ is distance at most $r_0 + (r_0/2)$ away from $\widetilde{\eta}$. 

We now bound the area of $N(\widetilde{\eta}, \frac{3r_0}{2})$. 
Consider a set of points $\{x_1, x_2, ..., x_m\}$ on a subsegment of $\widetilde{\eta}$ of length $L$, with the points positioned so that $d_{\widetilde{X}}(x_{j - 1}, x_j) = r_0/2$ for $j = 2, 3, ..., m$. 
Note that $N(\widetilde{\eta}, \frac{3r_0}{2})$ can be covered by $m = \lceil 2L/r_0 \rceil + 1$ balls in the set $\{B(x_j, 2r_0)\}_{j = 1}^m$. 
See \Cref{fig:neighborhoodcover}.

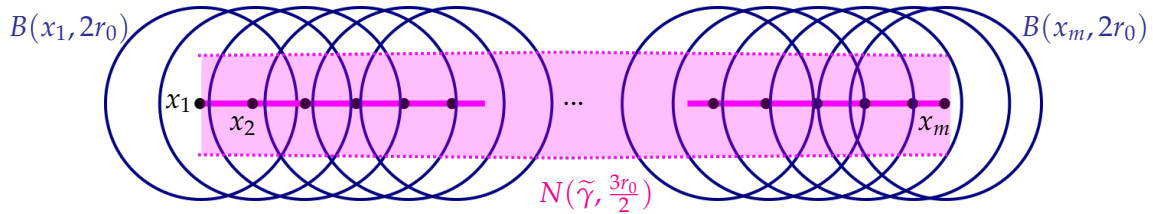
\begin{figure}[H]
    \centering
    \begin{tikzpicture}
    \node at (0, 0) {\includegraphics[width=0.8\linewidth]{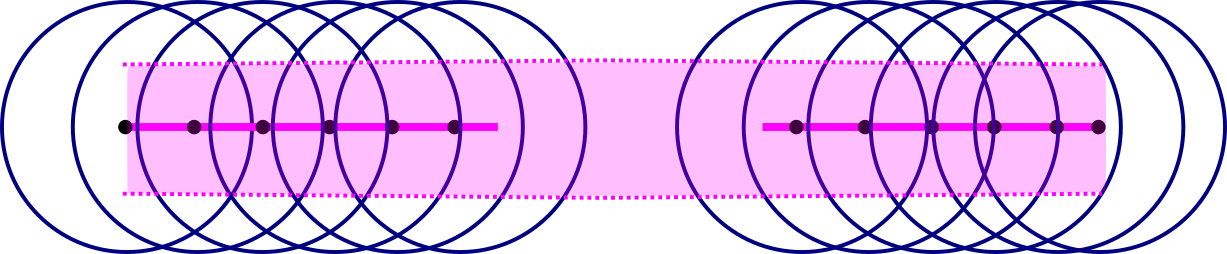}};
    \node at (0.3, -1.2) [Magenta]{$N(\widetilde{\gamma}, \frac{3r_0}{2})$}; 
    \node at (0, 0) {...};
    \node at (-5.25, 0) {$x_1$};
    \node at (-4.4, -0.3) {$x_2$}; 
    \node at (-6.7, 1) [Blue] {$B(x_1, 2r_0)$}; 
    \node at (6.8, 1) [Blue] {$B(x_m, 2r_0)$}; 
    \node at (4.8, -0.3) {$x_m$}; 
    \end{tikzpicture}
    \caption{Estimating the area of $N(\tilde{\gamma}, 3r_0/2)$ using balls of radius $2r_0$. Note that these are \emph{not} the same balls in the original covering $\{B(p_i, r_0)\}$.}    
    \label{fig:neighborhoodcover}
\end{figure}

Since $B(x_j, 2r_0) \stackrel{\text{homeo}}{\cong} T_j \times (0, 1)$ is embedded in $\widetilde{X}$, we can use the same argument from Step 2 to give an upper bound on the number of leaves of $T_j$, and thus on the area of $B(x_j, 2r_0)$:
\[
\Area\big(B(x_j, 2r_0)\big) \leq 
\left[\frac{B}{2(2r_0)}\left(\frac{3A}{2\pi} - 2\right) + 2\right]\left(4\pi\sinh^2(r_0)\right).
\]

Since $N(\widetilde{\eta}, 3r_0/2)$ is covered by $m = \left\lceil 2L/r_0 \right\rceil + 1$ balls $B(x_j, 2r_0)$, we have
\begin{equation}
\Area\big(N(\widetilde{\eta}, 3r_0 / 2)\big) \leq 
\left[\frac{B}{2(2r_0)}\left(\frac{3A}{2\pi} - 2\right) + 2 \right]\left(4\pi\sinh^2(r_0)\right)\left(\left\lceil \frac{2L}{r_0}\right\rceil + 1\right).
\end{equation}

Note that $\Gamma \neq \emptyset$, so $\sys(X) \leq B$. Moreover, since $r_0 \leq \sys(X)/4$, we have $2r_0 \leq \sys(X) \leq B$, and hence $B/(2r_0) \geq 1$. Therefore,
\[
\left[\frac{B}{2(2r_0)}\left(\frac{3A}{2\pi} - 2\right) + 2\right]
\leq 
\frac{3BA}{4\pi r_0}.
\]

It follows that
\begin{align}
\Area\big(N(\widetilde{\eta}, 3r_0 / 2)\big) 
&\leq 
\left(\frac{3BA}{4\pi r_0}\right)\left(4\pi\sinh^2(r_0)\right)\left(\left\lceil \frac{2L}{r_0}\right\rceil + 1\right) \notag\\
&\leq 
\left(\frac{3BA(2L + 2r_0)}{r_0^2}\right)\sinh^2(r_0). \label{eqn:stripbound}
\end{align}

We return our attention to the balls $\{B(p_k, r_0)\}$ that we initially covered $\widetilde{\eta}$ with; recall these are lifts of balls in $\{B(p_i, r_0)\}$, which cover $X$. Since $\{B(p_k, r_0/2)\}$ is a disjoint collection of balls of area at least $4\pi \sinh^2(r_0/4)$ all lying within $N(\widetilde{\eta}, 3r_0 / 2)$, we can derive an upper bound for $\beta(L)$, the number of balls in the set $\{B(p_k, r_0)\}_{k = 1}^{\beta(L)}$:
\begin{align}\label{eqn:B(L)bound}
\beta(L) 
&\leq \left(\frac{3B\,A(2L + 2r_0)}{4\pi r_0^2}\right)
\overbrace{\left(\dfrac{\sinh^2(r_0)}{\sinh^2(r_0/4)}\right)}^{16\cosh^2(r_0/4)\cosh^2(r_0/2)} \notag \\
&= \left(\dfrac{12B\,A}{\pi r_0^2}\right)\big(\cosh^2(r_0/4)\big)\big(\cosh^2(r_0/2)\big)(2L + 2r_0).
\end{align}

\noindent \textbf{Step 4: Derive an upper bound for $\#\G_X(L)$.} 
We can now use the previous steps to derive a bound for $\#\G_X(L)$. First, we bound the number of balls adjacent to each ball in $\{B(p_i, r_0)\}$, which covers $X$. In the case where $B(p_i, r_0)$ is not homeomorphic to a disk, each leaf of $T_i$ corresponds to a complementary half-plane in the universal cover from which a geodesic passing through $B(p_i, r_0)$ may enter or exit. Thus, once $\eta$, a closed geodesic of length at most $L$, enters some ball $B(p_i, r_0)$ that is homeomorphic to the product of a tree and an interval, \eqref{eqn:leaves} gives an upper bound on the number of directions of travel from $B(p_i, r_0)$ into adjacent balls. On the other hand, when $\eta$ enters a ball $B(p_i, r_0)$ which is homeomorphic to a disk in $\mathbb{H}^2$, by \Cref{cor:neighborsbound}, there are at most $13$ neighboring balls that $\eta$ may enter after exiting $B(p_i, r_0)$. Letting $\lambda$ be the maximum number of balls neighboring any given ball $B(p_i, r_0)$, we have
\begin{equation}\label{eqn:neighbors} 
\lambda \leq \max\left\{13,\ \left(\frac{B}{2r_0}\right)\left(\frac{3A}{2\pi} - 2\right) + 2\right\}.
\end{equation}

Therefore,
\begin{equation}\label{eqn:Lbound}
\#\G_X(L) \leq N\lambda^{\beta(L)},
\end{equation}
where bounds for $N$, $\beta(L)$, and $\lambda$ are given by \eqref{eqn:numberofballsbound}, \eqref{eqn:B(L)bound}, and \eqref{eqn:neighbors} respectively. Indeed, $\lambda^{\beta(L)}$ bounds the number of sequences of balls that a geodesic starting at a fixed initial ball may travel through, while $N$ accounts for the maximum number of choices of initial ball.

\noindent \textbf{Deducing an overall bound.}
Recall that $r_0 := \min\left\{\frac{\log(3)}{4},\ \frac{\sys(X)}{4}\right\}$ and $A:=\Area(X)$. 
By \eqref{eqn:numberofballsbound} and \eqref{eqn:Lbound}, we have $\#\G_X(L)\le N\lambda^{\beta(L)}$ with
\[
N\le \frac{4A}{\pi r_0^{2}}.
\]
Moreover, from \eqref{eqn:neighbors} we may use the coarse bound
\[
\lambda \le 15 + \frac{3AB}{4\pi r_0}.
\]
Finally, combining \eqref{eqn:B(L)bound} with the inequality $r_0\le \log(3)/4$ (so that $\cosh(r_0/4)$ and $\cosh(r_0/2)$ are uniformly bounded), we obtain the explicit estimate
\[
\beta(L)\le \frac{25\,AB}{\pi r_0^{2}}(L+r_0).
\]

Putting these together yields the following uniform upper bound, valid for all $L\ge 0$:
\begin{equation}\label{eqn:uniformupperbound}
\#\G_X(L)\ \le\ \frac{4A}{\pi r_0^{2}}\left(15+\frac{3AB}{4\pi r_0}\right)^{\frac{25\,AB}{\pi r_0^{2}}(L+r_0)}.
\end{equation}

This is the main result in Theorem \ref{thm:mainupper} from the introduction. Observe that the statement concerning the entropy, again in Theorem \ref{thm:mainupper}, follows from this result and a direct limit computation.

\begin{remark}
For fixed $A$ and uniform lower bound on the systole, the above bound gives an upper bound of the type $\left(K_1 B\right)^{K_2 B}$ for constants $K_1, K_2$ that only depend on systole and area. Letting $K:= K_2 \log K_1$, this becomes $e^{K B \log(B)}$. In Section \ref{sec:lower} will see that it is impossible to hope for a bound of this type without imposing a lower bound on systole. 
\end{remark}

\noindent \textbf{The upper bound in particular regimes.} We now analyze the behavior of the upper bound on $\#\G_X(L)$ in the regime where $\sys(X)\to 0$, $L\to\infty$, and $B\to\infty$. In this case we have $r_0=\sys(X)/4$, so simplifying (\ref{eqn:uniformupperbound}), we can find constants $K_0, K_1$, and $K_2$ depending only on $A$ such that \[
\#\G_X(L) \precsim \frac{K_0 A}{\sys^2(X)}
\left[\left(\frac{K_1B}{\sys(X)}\right)^{K_2\left(\frac{BL}{\sys^2(X)}\right)}\right].
\] 
For sufficiently large $B$ and $L$, we can further simplify the expression in terms of a constant $K$ which depends only on $A$: \[
\#\G_X(L) \precsim 
\left(\frac{B}{\sys^3(X)}\right)^{K\left(\frac{BL}{\sys^2(X)}\right)}.
\]

We now turn our attention to lower bounds on $\#\G_X(L)$, showing that dependence in boundary length and systole is necessary. We begin with a general construction of a family of surface amalgams with growing pasting curves. The second part is a more explicit example, and which shows that if we allow the systole to go to $0$, we get different behaviors. 

\subsubsection{Long gluing curves on fixed surfaces}

The general idea we will try to illustrate in this part is that having long gluing curves can give rise to many closed geodesics. It is based on a very general construction that uses fixed hyperbolic surfaces and then pastes them together.

Here is a general construction. Let $S$ be a closed hyperbolic surface of genus $g\geq 2$, $\beta$ a closed curve on $S$ and $m\geq 2$ a natural number. We let $X=X_{m,\beta}$ be the surface amalgam obtained by pasting $m$ copies of $S$ isometrically along $\beta$. The lower bound estimates will be for the simplest case, that is when $m=2$, because when $m$ is larger, $X$ will have even more curves. 

The main goal is to prove the following proposition. It uses the constant 
\[
c_\beta = \frac{\ell(\beta)}{2\pi^2(g-1)}
\]
whose origin will become transparent in the proof. 

\begin{proposition}\label{prop:longbeta}
Let $C>0$ be any constant such that $C< \log(2) \, c_\beta$. For large enough $L$, there are at least $e^{C L}$
primitive closed geodesics on $X$.
\end{proposition}

Before passing to the proof we observe that the proposition implies that by choosing $\beta$ long enough, the number of geodesics of length at most $L$ grows at least like $e^{CL}$ for any $C>0$. In particular, any bound on the exponential growth of the number of geodesics of a surface amalgam must contain a term that goes to infinity with respect to the length $B$ of the gluing curve. 

The rest of this part is dedicated to setting up the proof of this proposition. 

We begin by observing the following.

\begin{observation}\label{obs:project}
    Every closed geodesic on $X$ canonically projects to a closed geodesic on $S$. 
\end{observation}

This provides a natural surjection from the set of closed geodesics of $X$ to those of $S$. The key point is to quantify to what extent this maps fails to be injective. This comes from when a closed geodesic on $S$ crosses $\beta$: indeed at every crossing, the geodesic can "choose" a chamber on $X$.

More specifically, let $\alpha$ be a closed geodesic on $S$ with $i(\alpha, \beta)= k >1$. Then we claim that that there are $2^{k-1}$ at least closed geodesics on $X$ that project to $\alpha$ (and exactly $2^{k-1}$ when $m=2$).

Up until now, all of these comments are general, but we will make use of these observations in the following more specific construction which makes use of basic ergodic theory. 

For $L>0$, let $\G_S(L):=\{\gamma\subset S:\ \gamma \text{ primitive closed geodesic and } \ell(\gamma)\le L\}.$ Our next main ingredient is the following result about curves in $\G_S(L)$ and their intersection with a fixed curve. It is well-known to experts, possibly expressed in slightly different terms, and we provide a proof for completeness. 

\begin{lemma}\label{lem:expected-intersection}
Let $S$ be a closed hyperbolic surface of genus $g\ge 2$ and let
$\beta\subset S$ be a closed geodesic of length $\ell(\beta)$.

Then the expected intersection number between a geodesic in $\G_S(L)$ and $\beta$ satisfies
\[
\mathbb E_{\gamma\in \G_S(L)}\big[i(\beta,\gamma)\big]\ \sim\ c_\beta\,L
\qquad \mbox{ as }L\to\infty,
\]
where
\[
c_\beta=\frac{4\,\ell(\beta)}{\vol(T^1S)}=\frac{\ell(\beta)}{2\pi^2(g-1)}.
\]
\end{lemma}

\begin{proof}
The proof we provide uses Bonahon's currents \cite{Bonahon1988}, equidistribution and the prime geodesic theorem. 

Consider the space of geodesic currents $\mathcal C(S)$ on $S$ and the continuous bilinear intersection form
\[
i(\cdot,\cdot):\mathcal C(S)\times\mathcal C(S)\to \mathbb R_{\ge 0}
\]
which extends the intersection number of closed curves. Let $\mathcal L$ denote the Liouville current
associated to the hyperbolic metric, and let $\delta_\beta$ be the current associated to $\beta$.
Bonahon \cite{Bonahon1988} proves (for the standard normalization) that 
\[
i(\mathcal L,\delta_\beta)=\ell(\beta). 
\]

Thinking of $\mathcal L$ as the Liouville measure on the space of (unoriented) geodesics, the identity
$i(\mathcal L,\delta_\beta)=\ell(\beta)$ is exactly the Crofton-type statement that the
$\mathcal L$-mass of geodesics crossing $\beta$ is $\ell(\beta)$. 

We now use the equidistribution of primitive closed geodesics.
A theorem of Bowen \cite{Bowen1972} (see also Margulis' work on periodic orbit asymptotics \cite{Margulis1970}) implies that primitive periodic orbits of the geodesic flow become equidistributed with respect to the Bowen-Margulis measure. Since $S$ has constant curvature $-1$, the Bowen--Margulis measure coincides with the Liouville measure. By Margulis' equidistribution result, in the weak-* topology, we have the following convergence (see \cite{ErlandssonSouto2022}):

\[
\frac{1}{\sum_{\gamma\in \G_S(L)}\ell(\gamma)}
\sum_{\gamma\in \G_S(L)}\delta_\gamma
\ \longrightarrow\ 
\frac{4}{\vol(T^1S)}\,\mathcal L
\qquad\text{as }L\to\infty.
\]

We now apply the continuous linear functional $C\mapsto i(C,\delta_\beta)$ to the convergence above.
Using bilinearity and continuity of $i(\cdot,\cdot)$, we obtain
\[
\frac{1}{\sum_{\gamma\in \G_S(L)}\ell(\gamma)}
\sum_{\gamma\in \G_S(L)} i(\delta_\gamma,\delta_\beta)
\ \longrightarrow\ 
\frac{4}{\vol(T^1S)}\, i(\mathcal L,\delta_\beta)
=
\frac{4\,\ell(\beta)}{\vol(T^1S)}.
\]
Since $i(\delta_\gamma,\delta_\beta)=i(\beta,\gamma)$,
this is the averaged asymptotic
\begin{equation}\label{eq:avg-rate}
\frac{\sum_{\gamma\in \G_S(L)} i(\beta,\gamma)}{\sum_{\gamma\in \G_S(L)}\ell(\gamma)}
\ \longrightarrow\ 
\frac{4\,\ell(\beta)}{\vol(T^1S)}.
\end{equation}

The prime geodesic theorem \cite{buser} for compact hyperbolic surfaces, first proved by Huber \cite{huber}, states
\[
\#\G_S(L) \sim \frac{e^{L}}{L}
\]
for $L \to \infty$. A standard partial summation argument then gives
\[
\sum_{\gamma\in \G_S(L)}\ell(\gamma)\sim e^{L},
\]
and hence
\[
\frac{1}{\#\G_S(L)} \sum_{\gamma\in \G_S(L)}\ell(\gamma)\sim L.
\]
Now we can pass to expectations. We multiply \eqref{eq:avg-rate} by the average length to obtain the final result:
\begin{align*}
\mathbb E_{\gamma\in \G_S(L)}[\,i(\beta,\gamma)\,]
&=
\frac{1}{\#\G_S(L)}\sum_{\gamma\in \G_S(L)} i(\beta,\gamma) \\
&=
\left(\frac{\sum_{\gamma\in \G_S(L)} i(\beta,\gamma)}{\sum_{\gamma\in \G_S(L)}\ell(\gamma)}\right)
\left(\frac{\sum_{\gamma\in \G_S(L)}\ell(\gamma)}{\#\G_S(L)}\right)
\sim
\frac{4\,\ell(\beta)}{\vol(T^1S)}\,L.
\end{align*}
\end{proof}

\begin{remark}
As stated previously, the above result is not new. In particular it can be seen as a consequence of an asymptotic intersection law proved in the book of Erlandsson-Souto \cite[Proposition 2.4]{ErlandssonSouto2022}. Indeed, Erlandsson-Souto show (for random primitive closed geodesics $\gamma$) that
\[
\frac{i(\beta,\gamma)}{\ell(\gamma)} \longrightarrow \frac{\ell(\beta)}{2\pi^2(g-1)}
\]
so that typically $i(\beta,\gamma)\sim \frac{\ell(\beta)}{2\pi^2(g-1)}\,\ell(\gamma)$. In our setting, we choose $\gamma$ uniformly from the set $\mathcal G_S(L)$. Now since $\mathbb E[\ell(\gamma)]\sim L$ by the prime geodesic theorem, this yields
\[
\mathbb E_{\gamma\in\mathcal G_S(L)}[\,i(\beta,\gamma)\,]
\sim
\frac{\ell(\beta)}{2\pi^2(g-1)}\,L.
\]
\end{remark}

We will use the following corollary of \Cref{lem:expected-intersection}. 

\begin{corollary}\label{cor:many}
Let $S$ be a closed hyperbolic surface and let $\beta\subset S$ be a closed geodesic.
Let $c_\beta=\ell(\beta)/\vol(T^1S)$ be as in Lemma~\ref{lem:expected-intersection}.
Then for every $\varepsilon\in(0,1)$ there exists $L_0=L_0(\varepsilon)$ such that for all $L\ge L_0$
there are at least $e^{(1-\varepsilon)L}$ primitive closed geodesics $\gamma$ with $\ell(\gamma)\le L$
satisfying
\[
i(\beta,\gamma)\ \ge\ (1-\varepsilon)c_\beta L.
\]
\end{corollary}

\begin{proof}
Fix $\varepsilon\in]0,1[$. By Lemma~\ref{lem:expected-intersection}, for large enough $L$, we have
\[
\mathbb E_{\gamma\in \G_S(L)}[\,i(\beta,\gamma)\,]\ \ge\ (1-\varepsilon/2)c_\beta L.
\]
Now consider the set of geodesics we "don't want", namely the curves in the set 
\[
\Bad(L):=\left\{\gamma\in \G_S(L):\ i(\beta,\gamma)<(1-\varepsilon)c_\beta L\right\}.
\]
and their complement 
\[
\Good(L):=\left\{\gamma\in \G_S(L):\ i(\beta,\gamma)\geq(1-\varepsilon)c_\beta L\right\}.
\]
Let
$$p_L:= \frac{\#\Bad(L)}{\#\G_S(L)}$$
be the proportion of curves in our "bad" set. We will use a crude bound for the expected value for the "good" set, just by using the collar lemma. If $i(\gamma,\beta) = k>0$, then $\ell(\gamma) > 2 w(\beta)$ where $w(\beta)$ is the collar width around $\beta$. Hence, for any $\gamma \in \G_S(L)$, we have 
\[
i(\beta,\gamma) < \frac{1}{2 w(\beta)} L.
\]
Now 
\begin{align*}
\mathbb E_{\gamma\in \G_S(L)} [i(\beta,\gamma)] &= \frac{1}{\#\G_S(L)}\left(\sum_{\gamma \in \Good(L)}i(\beta,\gamma) + \sum_{\gamma \in \Bad(L)}i(\beta,\gamma) \right)\\
&< \frac{1}{\#\G_S(L)}\left( \#\Bad(L) \left(1-\varepsilon\right) c_\beta L + \#\Good(L) \frac{1}{2w(\beta)} L \right)\\
&= p_L  \left(1-\varepsilon\right) c_\beta L + (1-p_L) \frac{1}{2w(\beta)} L.
\end{align*}
Now we can divide by $L$, rearrange the terms and use our lower bound on expectation to get:
\[
\left(1-\frac{\varepsilon}{2}\right) c_\beta \leq \frac{1}{2w(\beta)} - \left(\frac{1}{2w(\beta)}- (1-\varepsilon)c_\beta\right)p_L.
\]
We note that $\frac{1}{2w(\beta)} > (1-\varepsilon)c_\beta$, otherwise the expected value of intersection is less than $(1-\varepsilon)c_\beta L$ which contradicts our very first assertion. We can now isolate $p_L$, which, by the previous sentence, has positive numerator and denominator: 
\[
p_L \leq \frac{\frac{1}{2w(\beta)}-\left(1-\frac{\varepsilon}{2}\right)c_\beta}{\frac{1}{2w(\beta)}-\left(1-\varepsilon\right)c_\beta}  < 1.
\]
Note that $p_L$ is bounded above by a function independent of $L$ and is strictly less than $1$. In particular, the "good" geodesics make up a definite portion $q_L := 1-p_L>0$ of all of the geodesics. We now take sufficiently large $L$ and use the prime geodesic theorem to ensure that there are at least $e^{(1-\varepsilon/2)L}$ closed geodesics in $\G_S(L)$. Thus we have 
\[
\#\Good(L) \geq q_L \#\G_S(L) \geq q_L e^{\left(1-\frac{\varepsilon}{2}\right)L}. 
\]
Now, again for sufficiently large $L$, the ratio between $q_L$ and $e^{\frac{\varepsilon}{2}L}$ is bounded above by $1$, and hence
$$
\#\Good(L) \geq e^{(1-\varepsilon)L}
$$
which proves the claim.
\end{proof}

We can now pass to how this affects the number of curves on $X$. The point is, for large $L$, most curves on $S$ of length up until $L$ will spend a lot of time intersecting $\beta$ and the longer $\beta$ is, the more intersections there are. As observed above, this gives rise to more curves on $X$. We can now finalize the proof of the proposition. 

\begin{proof}[Proof of Proposition \ref{prop:longbeta}.]
By Corollary \ref{cor:many}, for any small $1>\varepsilon>0$, and large enough $L$, there exist $e^{(1-\varepsilon)L}$ primitive closed geodesics $\gamma$ on $S$ with $\ell(\gamma)\le L$ and 
\[
i(\beta,\gamma)\ \ge\ (1-\varepsilon)c_\beta L.
\]
Each such curve gives rise to at least $2^{i(\beta, \gamma) -1}$ geodesics on $X$, hence there are at least 
\[
\frac{1}{2} \, e^{(1-\varepsilon)L} \, 2^{(1-\varepsilon)c_\beta L} = \frac{1}{2} \, e^{(1-\varepsilon)L + \log(2)(1-\varepsilon)c_\beta L}
\]
which, for large enough $L$, is greater than 
\[
e^{\log(2)(1-\varepsilon)c_\beta L}.
\]
This proves the proposition. 
\end{proof}

\subsubsection{Short systoles and high entropy}

In this section, we construct a family of surface amalgams with a large number of closed geodesics as the length grows. To do so, we begin the construction of a one-holed torus, which will be the basic building block of the surface amalgam. 

\noindent \textbf{The torus $T_{b,s}$.}  
Throughout, we consider positive constants $B,s,b>0$ which satisfy inequalities that we will make precise. They will have geometric meanings: $s$ will be the systole of the torus (and later of the surface amalgam), and $b$ will be the boundary length of the one-holed torus. As before, $B$ is the length of the gluing geodesic; in this construction, $B$ will end up depending on $s$.

We consider a right-angled pentagon with non-adjacent sides of lengths $\frac{s}{2}$ and $\frac{b}{4}$. We then use four copies of it, pasted together so that the corners of the pentagon meet, to construct a one-holed torus as in Figure \ref{fig:pentagon}. This results in a one-holed torus with both a front-to-back and a top-to-bottom symmetry. The curve $\alpha$ is of length $s$, and it intersects a curve $\beta$ at a right angle, where $\beta$ is a concatenation of sides of the pentagons. The boundary of the one-holed torus is of length $b$. By Theorem 2.3.4 of \cite{buser}, the lengths of $\alpha$ and $\beta$ satisfy 

\begin{equation} \label{eqn:pentagon}
\sinh\left(\frac{\ell(\alpha)}{2}\right) \sinh\left(\frac{\ell(\beta)}{2}\right)= \cosh\left(\frac{b}{4}\right).
\end{equation}

\begin{figure}[H]
    \centering
    \begin{tikzpicture}
    \node at (0,0) {\includegraphics[width=0.8\linewidth]{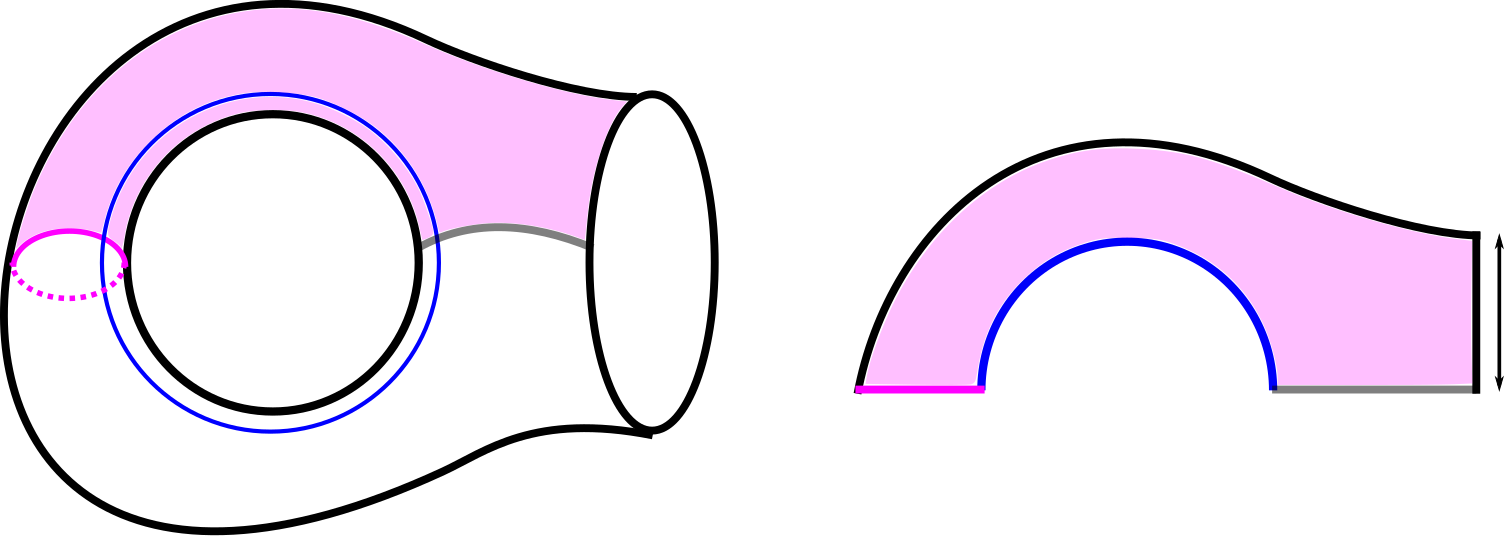}};
    \node at (-6.5, 0) [Magenta] {$\alpha$};
    \node at (-2.5, -1) [blue] {$\beta$}; 
    \node at (6.6, -0.3) {$\frac{b}{4}$};
    \end{tikzpicture}
    \caption{A one-holed torus $T_{b, s}$ (left) constructed from four copies of a right-angled pentagon (left). Its boundary length $b$ is given by (\ref{eqn:pentagon}).}
    \label{fig:pentagon}
\end{figure}
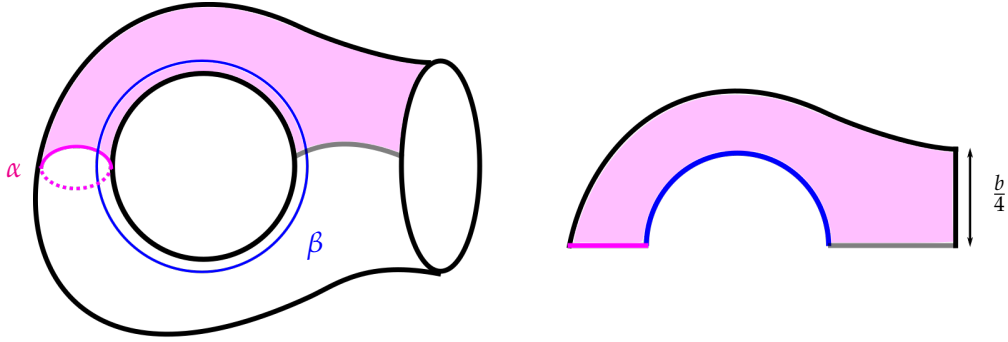

We will be interested in the case when $s \to 0$ and $b$ is small but fixed. We can suppose that $s<b$ and $b\leq 2\arcsinh(1)$. In this case, since $\alpha$ intersects all interior curves of the torus and since $\ell(\alpha) < 2\arcsinh(1)$, by the collar lemma, all other curves (including $\beta$) have length $> 2\arcsinh(1)$. Now since $\ell(\alpha) <b$, $\alpha$ is the systole of $T_{b,s}$ and so 
$$
\sys(T_{b,s})=s.
$$

\noindent \textbf{The surface amalgam $X_{b,s,k}$.} 
We fix an orientation of $\alpha$ and, for any $k\in \N$, obtain a curve $\beta_k$ obtained by applying $k$ positive Dehn twists to $\beta$ about $\alpha$. In particular, $\beta_0:=\beta$. Note that $\beta$ and $\beta_k$ intersect exactly $k$ times.

The length of the curve $\beta_k$, by using a standard hyperbolic trigonometry formula in the triangle (c.f. \cite{buser} Theorem 2.2.2 and \Cref{fig:righttriangle}), satisfies
$$
\cosh\left(\frac{\ell(\beta_k)}{2}\right) = \cosh\left(\frac{k s}{2}\right) \cosh\left(\frac{\ell(\beta)}{2}\right).
$$

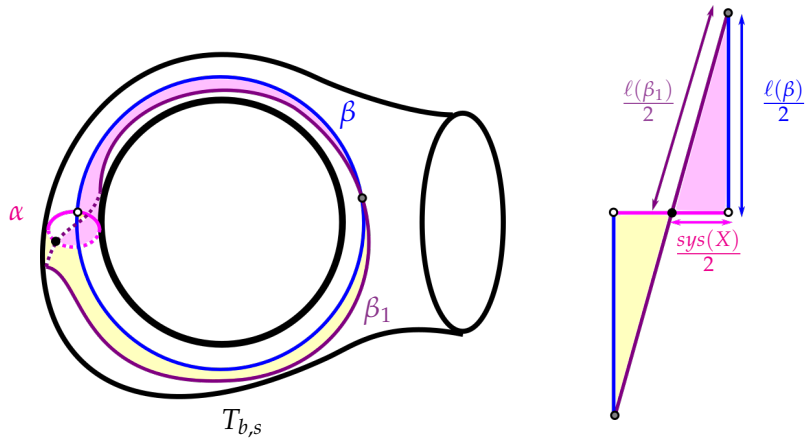
\begin{figure}[H]
    \centering
    \begin{tikzpicture}
    \node at (0, 0) {\includegraphics[width=0.6\linewidth]{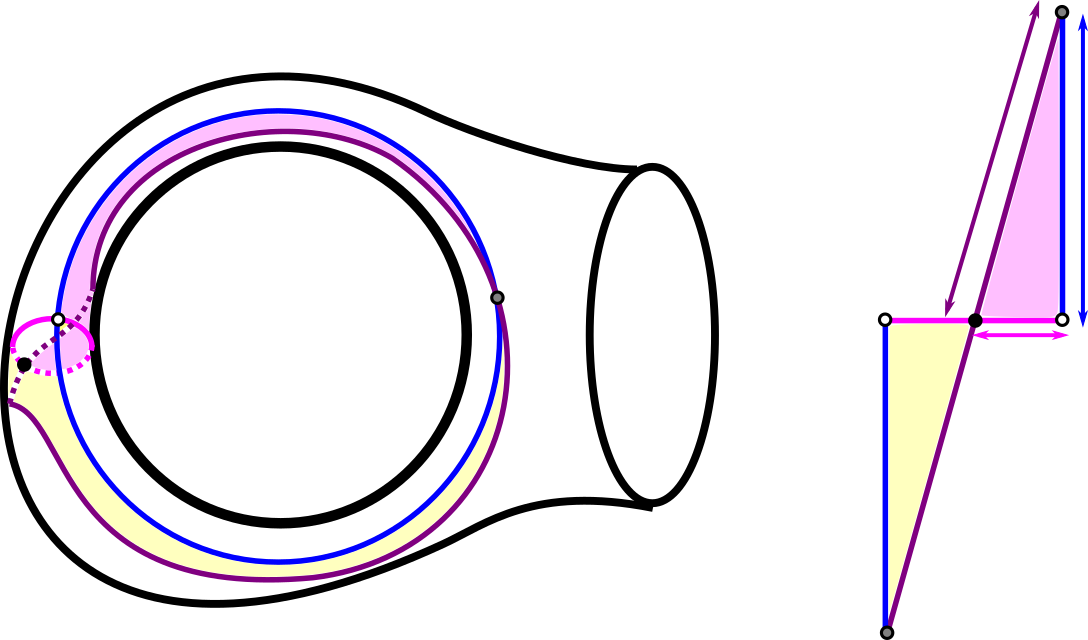}};
    \node at (-0.6, 1.3) [blue]{$\beta$};
    \node at (-0.2, -1.3) [Fuchsia]{$\beta_1$}; 
    \node at (-5, 0) [Magenta]{$\alpha$};
    \node at (-2, -2.8) {$T_{b, s}$};
    \node at (4.2, -0.5) [Magenta] {$\frac{sys(X)}{2}$};
    \node at (5.2, 1.5) [blue] {$\frac{\ell(\beta)}{2}$}; 
    \node at (3.4, 1.5) [Purple] {$\frac{\ell(\beta_1)}{2}$};
    \end{tikzpicture}
    \caption{Cutting along $\beta_1 = T_{\alpha}(\beta)$, $\alpha$, and $\beta$ yields two right triangles.}
    \label{fig:righttriangle}
\end{figure}

Now we construct a surface amalgam $X_{b,s,k}$ by taking two copies of $T_{b,s}$ and pasting them along both $\beta_k$ and their boundary curve of length $b$, both times via the identity. 

We observe that the resulting surface amalgam has systole of length $s$ and the pasting length is equal to $B:=\ell(\beta_k) $. This is because when we paste the boundary components of length $b$ together, we do not introduce a singular curve. (There is in fact only one chamber of the surface amalgam which is a four-holed sphere.)

Using the formula above:
$$
B= 2 \arccosh\left( \cosh\left(\frac{k s}{2}\right) \cosh\left(\frac{\ell(\beta)}{2}\right)\right). 
$$

We note that the geodesic $\beta$ on $T_{b,s}$ lifts on $X_{b,s,k}$ to at least $ 2^k$ curves, all of length $\ell(\beta)$.Indeed, every time $\beta$ intersects the gluing curve $\beta_k$, there are at least $2$ different directions a lift of $\beta$ can proceed in without backtracking. However, we shall do "better" by considering lifts of another curve. 

\noindent \textbf{Constructing many curves on $X_{b,s,k}$.} 
We return to $T_{b,s}$. Instead of twisting $\beta$ around $\alpha$, we consider the curves obtained by twisting $\alpha$ along $\beta$ $k'\in \N$ times, again for a fixed orientation of $\beta$. We call these curves $\alpha_{k'}$. We can again compute their length exactly, like for $\beta_k$.

We will mainly need the following observation:

\begin{lemma}\label{lem:intersections} For $k,k'\in \N$ we have
$$
i(\alpha_{k'}, \beta_k)= k' k +1.
$$
\end{lemma}
\begin{proof}
We use a homology argument. Since $\alpha$ and $\beta$ intersect once, their homology classes form a basis of $H_1(S;\mathbb Z)$.
Let $\langle\cdot,\cdot\rangle$ be algebraic intersection; with the chosen orientations
$\langle\alpha,\beta\rangle=1$, hence $\langle\beta,\alpha\rangle=-1$, and
$\langle\alpha,\alpha\rangle=\langle\beta,\beta\rangle=0$.

By definition, $\alpha_k$ is obtained from $\alpha$ by $k$ positive twists along $\beta$, so its class is
\[
[\alpha_k]=[\alpha]+k[\beta].
\]
Likewise, $\beta_{k'}$ is obtained from $\beta$ by $k'$ positive twists along $\alpha$, hence
\[
[\beta_{k'}]=[\beta]-k'[\alpha].
\]
Therefore
\[
\langle[\alpha_k],[\beta_{k'}]\rangle
=\langle[\alpha]+k[\beta],[\beta]-k'[\alpha]\rangle
=\langle\alpha,\beta\rangle-k'\langle\alpha,\alpha\rangle
+k\langle\beta,\beta\rangle-kk'\langle\beta,\alpha\rangle
=1+kk'.
\]
For $k,k'\in\mathbb N$ this is positive, so $i(\alpha_k,\beta_{k'})=1+kk'$.
\end{proof}

For example, in \Cref{fig:intersections}, the pink curve of length $L$, $\alpha_7$, and the blue gluing curve $\beta_2$ intersect $15$ times. (One of the intersections, which occurs on the back of the surface, is not shown.) 

\begin{figure}[H]
    \centering
    \includegraphics[width=0.5\linewidth]{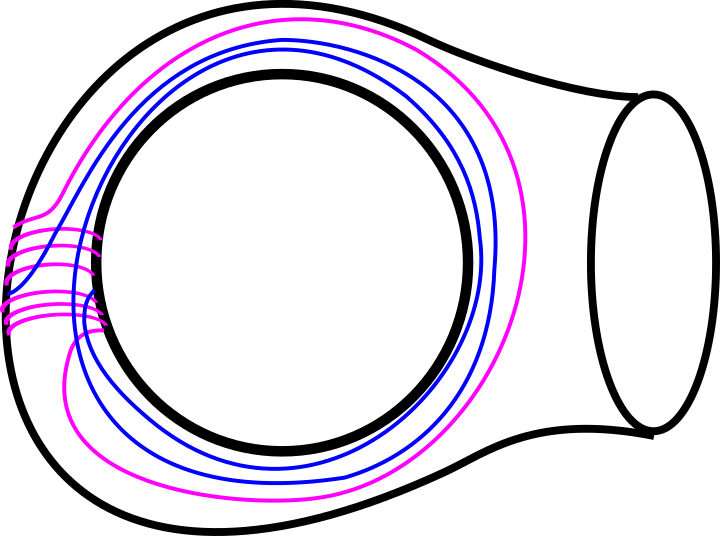}
    \caption{An example of \Cref{lem:intersections} where $k' = 7$ and $k = 2$.}
    \label{fig:intersections}
\end{figure}

We use an easy approximation of the length of $\alpha_{k'}$ via the triangle inequality:
$$
\ell(\alpha_{k'}) < k' \ell(\beta) + \ell(\alpha)= k' \ell(\beta) + s.
$$
Setting $L=k'\ell(\beta)+s$, using the same argument as before, $\alpha_{k'}$ lifts on $X_{b,s,k}$ to at least $2^{k \, k'+1}$ closed geodesics of length at most $L$.
To make this meaningful, we need to understand how these numbers evolve when we play around with the parameters. Before getting into the computations, we explain the general strategy.

\noindent {\bf Strategy.} First, we explain how to define $k$ so that $s$ and $k$ will both only depend on the length of the gluing curve $B$. We will fix $b$ and set $k$ to be roughly $1/s$.
This allows $\beta_k$ to wrap around $\alpha$ roughly $k$ times while adding only length $1$ to $\ell(\beta)$; that is, $B = \ell(\beta_k) \lesssim \ell(\beta) + 1$. The only remaining parameter for the construction of $X_{b,s,k}$ is now $s$ which will be going to $0$. Note that the length of $\beta$ is roughly twice the collar length of $\alpha$, which in turn is roughly $\log(1/s)$. The total pasting length $B$ is now, up to an additive constant, $2 \log(1/s)$. Hence $k$ becomes roughly $e^{B/2}$. 

Now we let $k'$ grow. The lengths of the curves $\alpha_{k'}$ grow linearly in $k'$, but with a factor in front that depends on $s$ and that can be estimated in terms of $B$. We can now put the results together: since there are at least $2^{k k' +1}$ curves of length $L(k')$, this results in having at least $2^{C\frac{e^{B/2}}{B}L}$ curves
as $L$ grows for some universal constant $C$. 

\noindent{\bf Implementing the strategy with explicit estimates.} Before digging into the other parameters, we now fix $b:=2\arcsinh(1)$. This implies 
\[
\cosh\!\left(\frac{b}{4}\right)
=
\sqrt{\frac{1+\sqrt{2}}{2}}.
\]
And from this, we have
\[
\ell(\beta)
=
2\,\operatorname{arcsinh}\!\left(
\frac{\sqrt{\frac{1+\sqrt2}{2}}}
{\sinh\!\left(\frac{s}{2}\right)}
\right).
\]

As we are interested in the regime where $s$ is small, we now make the further assumption that \(0<s\le \tfrac12\). Using this we can estimate $\ell(\beta)$. 

\begin{lemma}
For \(0<s\le \frac12\) we have
\[
2\log\!\left(\frac{2}{s}\right)\ \le\ \ell(\beta)\ \le\ 2\log\!\left(\frac{7}{s}\right).
\]
\end{lemma}

\begin{proof}
We use the elementary bounds
\[
x \le \sinh x \le x\cosh x \qquad (x\ge0)
\]
and
\[
\log t \le \operatorname{arcsinh}(t) \le \log(3t) \qquad (t\ge1).
\]

First, since \(\sinh(s/2)\ge s/2\),
\[
\ell(\beta)
\le
2\,\operatorname{arcsinh}\!\left(
\frac{2\sqrt{\frac{1+\sqrt2}{2}}}{s}
\right).
\]
Because \(s\le \tfrac12\), the argument is at least \(1\), so the second inequality above gives
\[
\ell(\beta)
\le
2\log\!\left(
\frac{6\sqrt{\frac{1+\sqrt2}{2}}}{s}
\right).
\]
Using elementary calculus (or direct numerical evaluation) one checks that
\(6\sqrt{\frac{1+\sqrt2}{2}}<7\), which yields
\[
\ell(\beta)\le 2\log\!\left(\frac{7}{s}\right).
\]

For the lower bound, since \(\sinh(s/2)\le (s/2)\cosh(s/2)\) and \(s/2\le 1/4\),
\[
\sinh(s/2)\le (s/2)\cosh(1/4),
\]
hence
\[
\ell(\beta)
\ge
2\,\operatorname{arcsinh}\!\left(
\frac{2\sqrt{\frac{1+\sqrt2}{2}}}{s\cosh(1/4)}
\right).
\]
Using \(\operatorname{arcsinh}(t)\ge \log t\) and simplifying the constants (again by a direct numerical check) gives
\[
\ell(\beta)\ge 2\log\!\left(\frac{2}{s}\right).
\]
\end{proof}

We set \(k=\lfloor 1/s\rfloor\) and from the previous lemma we get immediate upper and  lower bounds on $\ell(\beta_k)$ and hence on $B$.
\begin{lemma}
For \(0<s\le \frac12\) and \(k=\lfloor 1/s\rfloor\) we have 
\[
2\log\!\left(\frac{2}{s}\right)
\;\le\;
\ell(\beta_k)=B
\;\le\;
2\log\!\left(\frac{12}{s}\right).
\]
\end{lemma}
\begin{proof}
These estimates just follow from 
\[
\ell(\beta) \leq \ell(\beta_k) \leq \ell(\beta) + ks
\]
and then, because $ks \leq 1$, using the previous estimates, we have
\[
\ell(\beta) + ks \leq \ell(\beta) + 1 \leq 2\log\!\left(\frac{12}{s}\right).
\]
\end{proof}
From the upper bound on $B$, we obtain the following lower bound on $k$ in terms of $B$.

\begin{lemma}
For $s$ and $k$ as above, we have
\[
k \;\ge\; \frac{e^{B/2}}{12} - 1 .
\]
\end{lemma}

\begin{proof}
From the previous lemma we have
\[
B \le 2\log\!\left(\frac{12}{s}\right).
\]
Exponentiating gives
\[
e^{B/2} \le \frac{12}{s},
\]
so that
\[
s \le 12e^{-B/2}.
\]
Since $k=\lfloor 1/s\rfloor$, we have $k \ge \frac{1}{s}-1$. Substituting the
bound on $s$ yields
\[
k \ge \frac{e^{B/2}}{12}-1 .
\]
\end{proof}

We now take $L>0$ and set $k'$ to be equal to 
\[
k'=\max\{m\in\mathbb{Z}_{\ge 0}: m\,\ell(\beta)+s\le L\}.
\]
(This ensures that on $T_{b,s}$ the curves $\alpha_{k'}$ are of length at most $L$.) 
Then we have the following estimate (which is only meaningful for large enough $L$).

\begin{lemma} Suppose $0 < s \leq \frac{1}{2}$.
The quantity $k'$ satisfies
\[
k' \ge \frac{L-1}{B}-1.
\]
\end{lemma}

\begin{proof} By definition of $k'$, 
\[
k'=\left\lfloor \frac{L-s}{\ell(\beta)} \right\rfloor,
\]
and therefore
\[
k' \ge \frac{L-s}{\ell(\beta)}-1.
\]
By definition of $B$, we have $\ell(\beta)\le B$. Hence
\[
k' \ge \frac{L-s}{B}-1.
\]
Since $0<s\le \tfrac12$, we have $L-s \ge L-1$, and thus
\[
k' \ge \frac{L-1}{B}-1.
\]
\end{proof}

We now go back to $X_{b,s,k}$ and can prove Theorem \Cref{thm:B}. Since our surface amalgams no longer depend on $b$, and since we have estimated $s$ and $k$ in terms of $B$, we index them by their pasting length and set $X_B:=X_{b,s,k}$. We are only interested in the regime where $B$ is sufficiently large, so we set \(B \ge 2\log 24\), and our interest is in when $L$ grows so we suppose that $L \ge 2B+5$. 

\begin{proof}
Since
\[
k \ge \frac{e^{B/2}}{12}-1,
\qquad
k' \ge \frac{L-1}{B}-1,
\]
and we have supposed that
\[
B \ge 2\log 24,
\qquad
L \ge 2B+5,
\]
we can obtain the weaker inequalities:
\[
k \ge \frac{e^{B/2}}{24},
\qquad
k' \ge \frac{L-1}{2B}.
\]
From this we have
\[
kk' \ge \frac{(L-1)e^{B/2}}{48\,B}
\ge \frac{Le^{B/2}}{96\,B},
\]
and therefore
\[
2^{kk'+1} \;\ge\; 2^{\,1+\frac{1}{96}\frac{e^{B/2}}{B}\,L},
\]
which proves the theorem.
\end{proof}

Finally, observe that the statement concerning the entropy in Theorem \ref{thm:B} from the introduction follows from a direct computation. 

\bibliographystyle{alpha}
\bibliography{main}

\end{document}